\documentclass[journal,twoside,web]{ieeecolor}
\usepackage[utf8]{inputenc}
\usepackage{generic}
\usepackage{cite}
\usepackage{amsmath,amssymb,amsfonts}
\usepackage{tcolorbox}
\usepackage{algorithmic}
\usepackage{enumerate}
\usepackage{graphicx}
\usepackage{graphics}
\usepackage{url}
\usepackage{hyperref}
\usepackage{textcomp}
\usepackage{mathtools, cuted}
\usepackage{stfloats}
\usepackage{multirow}
\usepackage{subcaption}
\newcommand{\vectornorm}[1]{\vert\vert #1\vert \vert}
\def\deff{\stackrel{\triangle}{=}}
\newcommand{\R}{\mathbb{R}}
\newcommand{\D}{\mathcal{D}}
\newcommand{\K}{\mathcal{K}}
\newcommand{\X}{\mathcal{X}}

\newcommand{\xx}{\nabla_{xx}}
\newcommand{\xt}{\nabla_{xt}}
\newcommand{\x}{\nabla_x}
\newcommand{\st}{^{\star}}
\newtheorem{assumption}{Assumption}
\newtheorem{corollary}{Corollary}
\newtheorem{lemma}{Lemma}
\newtheorem{remark}{Remark}
\newtheorem{definition}{Definition}
\newtheorem{theorem}{Theorem}
\newtheorem{proposition}{Proposition}
\newcounter{example}[section]
\newenvironment{example}[1][]{\refstepcounter{example}\par\medskip
   \noindent \textbf{Example~\theexample. #1} \rmfamily}{\medskip}

\begin{document}
\title{Dynamical System Approach for Time-Varying Constrained Convex Optimization Problems}
\author{Rejitha Raveendran, Arun D. Mahindrakar, and Umesh Vaidya
\thanks{Rejitha Raveendran is with the Research and Innovation Group, Tata Consultancy Services, Bengaluru 560066, India (e-mail: raug31@gmail.com). }
\thanks{Arun D. Mahindrakar is with the Department of Electrical Engineering and a member of RBCDSAI, Indian Institute of Technology Madras, Chennai 600036, India (e-mail: arun\_dm@iitm.ac.in).}
\thanks{Umesh Vaidya is with the Department of Mechanical Enginnering, Clemson University, Clemson, SC 29634 USA (e-mail: uvaidya@clemson.edu).}}
%%%%%%%%%%%%%%%%%%%%%%%%%%%%%%%%%%%%%%%%%%%%%%%%%%%%%%%%%%%%%%%%%%%%%%%%%
\maketitle
%%%%%%%%%%%%%%%%%%%%%%%%%%%%%%%%%%%%%%%%%%%%%%%%%%%%%%%%%%%%%%%%%%%%%%%%%
\begin{abstract}
Optimization problems emerging in most of the real-world applications are dynamic, where either the objective function or the constraints change continuously over time.  This paper proposes projected primal-dual dynamical system approaches to track the primal and dual optimizer trajectories of an inequality constrained time-varying (TV) convex optimization problem with a strongly convex objective function. First, we present a dynamical system that asymptotically tracks the optimizer trajectory of an inequality constrained TV optimization problem. Later we modify the proposed dynamics to achieve the convergence to the optimizer trajectory within a fixed time. The asymptotic and fixed-time convergence of the proposed dynamical systems to the optimizer trajectory is shown via Lyapunov based analysis. Finally, we consider the TV extended Fermat -Torricelli problem (eFTP) of minimizing the sum-of-squared distances to a finite number of nonempty, closed and convex TV sets, to illustrate the applicability of the projected dynamical systems proposed in this paper. 
\end{abstract}
%%%%%%%%%%%%%%%%%%%%%%%%%%%%%%%%%%%%%%%%%%%%%%%%%%%%%%%%%%%%%%%%%%%%%%%%%
%%%%%%%%%%%%%%%%%%%%%%%%%%%%%%%%%%%%%%%%%%%%%%%%%%%%%%%%%%%%%%%%%%%%%%%%%
\section{Introduction}
    Optimization is a mechanism for selecting the best element from a set of choices that satisfy the stated criteria to utilize the available resources effectively.  Accordingly, the applicability of optimization is widespread among diverse fields, including guidance and navigation of aerial and autonomous vehicles, machine learning, computational system biology, molecular modelling, economics, etc. \cite{Appl:RobtNAvgn,Appl:OptimalPowerFlow, Appl:ZNN, Appl:SgnalPrcsng,Appl:ML}. Most of the literature on the solution approaches for an optimization problem assumes that the problem is static, that is, neither the objective function nor the constraints of the optimization problem change with respect to time.  However, most of the evolving optimization problems in real-time applications are dynamic \cite{TV_Formation_Trnkg, Appl:ResAll, Appl:Wireless}, that is either the objective function or the constraints vary according to some parameters.
    
    In a TV optimization problem, either the objective function or the constraints are time-dependent and continuously vary with time, resulting in an optimizer trajectory formed with the optimal points at each time instant.  Therefore, solving a TV optimization problem eventually becomes a tracking problem, where the optimizer trajectory of the problem has to be tracked. Even though there exist various approaches to solve time-invariant optimization problems, that includes Newton's method, subgradient methods, interior point method \cite{boyd}, and primal-dual dynamics \cite{Cherukuri:asymp_conv}, not much work has been carried out in the area of TV optimization. Batch algorithms, online optimization and prediction-correction algorithms are the commonly used techniques for solving TV convex optimization problems.
    
   In batch algorithms \cite{Simonetto:TimeStructured}, the TV optimization problem is sampled at a particular sampling time, and the sampled optimization problem is solved at each sampling instants. Hence it requires an extensive computation time and guarantees an optimal solution only if each sampled optimization problem converges within a time interval that is smaller than the sampling intervals. Therefore the batch optimization method is not widely accepted in real-time applications.  Online optimization technique \cite{Akbari:Online, DistributedOCO, Marcello:Online,OCO, OCO_Arxiv} is an unstructured algorithm \cite{Simonetto:TimeStructured} used for solving TV optimization problems. It estimates the optimal points at each time instants only with the available past information and without using properties of the optimization problem to be solved. Whereas the time-structured algorithms use the properties of the optimization problem (such as Lipschitzness, differentiability, etc.)  to track the optimizer trajectories; one of the popular time-structured algorithms is the prediction-correction method.  
    
    Prediction-correction algorithms are the frequently used approaches to track the optimizer trajectories of a TV convex optimization problem. It involves a ``prediction" step to predict the optimizer trajectory in the next instant, and a ``correction" step to drive the predicted value to the optimizer trajectory. However, the prediction-correction algorithm assumes strong convexity of the objective function to solve unconstrained \cite{Simonetto:AClassof,Simonetto:unconstrained, SImonetto_Confe} and constrained \cite{Simonetto:Constr1,Simonetto:Constr2,Simonetto:Constr3} TV convex optimization problems. A prediction-correction based continuous time dynamical system approach is proposed in \cite{Zhao} to solve unconstrained TV optimization problems, and guarantees the exponential convergence of the trajectories of the proposed system to the optimizer trajectory.
    Later, Fazlyab et al. extended such dynamical systems in \cite{PredCorr:Mehyar} by employing the associated Lagrangian function to track the optimizer trajectory of an equality constrained TV optimization problem. Inequality constrained TV optimization problems are also solved in \cite{PredCorr:Mehyar} by approximating it into an unconstrained optimization problem using logarithmic barrier functions with appropriately chosen barrier parameter and slack variables. In \cite{projectedPowerSystem}, the authors use projected dynamical system approach to track the optimizer trajectory of a TV nonconvex optimization problem, where the inequality constraints are eliminated by projecting the solution to the feasible region at each instant of time. The dynamical system approach proposed in this work employs a projection operator to guarantee the feasibility of the Lagrange multipliers.

    The main contributions of the work are as follows. We propose projected dynamical system approaches to track the optimizer trajectory of an inequality constrained convex optimization problem with TV objective function and TV inequality constraints, based on the KKT optimality conditions \cite{PredCorr:Mehyar} and the optimizer trajectory characterization derived from the KKT conditions. The proposed dynamical systems differ from each other in terms of their settling time function, the time at which the system trajectories converge to the optimizer trajectory of the TV optimization problem.
    In the first approach, the dynamical system achieves asymptotic convergence to the optimizer trajectory, whereas the second system ensures convergence within a fixed time.
    Unlike the method proposed in \cite{PredCorr:Mehyar} which employs the interior point method, the discontinuous dynamical system approach presented in this work employs the Lagrangian function associated with the optimization problem and a projection operator to ensure the feasibility of the Lagrange multipliers associated with the inequality constraints.  Moreover, the proposed dynamical system allows the Lagrange multiplier to shift from zero to a positive value whenever the corresponding inequality constraint switches from inactive to active mode. The asymptotic and fixed-time convergence of the proposed projected dynamical system to the optimizer trajectory is shown via Lyapunov-based analysis.  Throughout this paper, we assume that the objective function of the optimization problem is twice continuously differentiable and strongly convex, which is a rather typical assumption in most existing works. As an example, we formulate the TV version of the eFTP that solves for a point in the Euclidean space that minimizes the sum-of-distances to a finite number of nonempty, closed and convex TV sets and the corresponding points in each convex set with the assumption that at least one of the convex set is bounded. We approximate the objective function with the sum-of-squared distances to track the optimizer trajectory of TV eFTP using the proposed dynamical system approach. A dynamical system approach to solve the TV eFTP is also a new contribution. 
%%%%%%%%%%%%%%%%%%%%%%%%%%%%%%%%%%%%%%%%%%%%%%%%%%%%%%%%%%%%%%%%%%%%%%%%%
\section{Notations and Math Preliminaries}
    \subsection{Notations}
    Let $\R^n$ denote the $n$-dimensional Euclidean space and $\vectornorm{.}$ be the Euclidean norm in $\R^n$. Let $\R,\R_{\geq0}$ and $\R_{>0}$ be the set of real numbers, non-negative real numbers and positive real numbers respectively and the identity matrix in $\R^n$ is denoted by $I_n$. The gradient of a function $f(x,t):\R^n\times \R_{\geq 0}\rightarrow\R$ with respect to $x$ is denoted by $\x f(x,t)$ and its partial derivative with respect to $t$ is $\xt f(x,t)$. The time-derivative and Hessian of the function $f(x,t)$ is represented by $\nabla_t f(x,t)$ and $\xx f(x,t)$ respectively.  The row-wise multiplication of the matrix $Q\in\R^{m\times n}$ with a vector $p\in\R^m$, denoted by ``$\circ$", is defined as $p\circ Q = \begin{bmatrix}p_1q_1 &p_2q_2&\cdots&p_mq_m\end{bmatrix}$ where $q_1,q_2,\ldots, q_m$ are the rows of the matrix $Q$.  For a given $x\in\R^n$, $x_i$ represents the $i^{\textrm{th}}$ component of the vector $x$ and $B(x,r)$ denotes the open ball centered at $x$ with a radius $r>0$. For a set $A$, its cardinality is denoted by $|A|$ and the relative interior of the set $A$ is denoted with $\textrm{relint}(A)$ and is defined as $\{x\in A:B(x,r)\bigcap\textrm{ aff}(A)\subseteq A \textrm{ for some } r>0\}$, where $\textrm{aff}(A)$ is the set of all affine combinations of the elements in $A$. For a symmetric matrix $A, A\succ 0$ represents positive definiteness and $A\succcurlyeq 0$ for positive semi-definiteness. The notation $[b]_a^+$ denotes the component wise operation\\ $[b]_a^+=\begin{cases}b,\quad& \textrm{ if }a>0\\\max\{0,b\},\quad&\textrm{ if } a=0\end{cases}$ with $a\in\R_{\geq0}$ and $b\in\R$. The projection of a point $x\in\R^n$ onto a closed and convex set $\mathcal{W}\subset \R^n$ is defined as,
        $$\textrm{proj}_{\mathcal{W}}(x) = \arg \min_{y\in \mathcal{W}} \vectornorm{y-x}.$$ 
    The projection of a given vector $v\in\R^n$ at a given point $x\in\mathcal{W}$ with respect to the set $\mathcal{W}$ \cite{nagurney1995projected} is,
        $$\Pi_{\mathcal{W}}(x,v)=\lim_{\delta\rightarrow 0^+} \dfrac{\textrm{proj}_{\mathcal{W}}(x+\delta v)-x}{\delta}.$$
%%%%%%%%%%%%%%%%%%%%%%%%%%%%%%%%%%%%%%%%%%%%%%%%%%%%%%%%%%%%%%%%%%%%%%%%%
    \subsection{Fixed-time convergence}
    Consider a nonlinear TV dynamical system of the form 
    \begin{equation}\label{sys1}
        \dot{x}(t)=h(x,t),\enspace x(t_0)=x_0 \textrm{ and } t_0\geq0
    \end{equation}
    where $h:\D\times R_{\geq 0}\rightarrow\R$ and $\D\subseteq\R^n$ is the domain containing the origin. The function $h(x,t)$ is jointly continuous in both $x$ and $t$.
    Let  $\phi(x_0,t_0,t):\D\times\R_{\geq 0} \times\R_{\geq 0}\rightarrow \D $ be the flow associated with \eqref{sys1}. 
    
    \begin{definition}[\cite{Haddad:FTS_ACC}]
        The zero solution $x(t)\equiv0$ of the nonlinear dynamical system \eqref{sys1} is finite-time stable if there exists an open neighbourhood $\mathcal{N}\subseteq\D$ of the origin and a function $T:\mathcal{N}\backslash\{0\}\times[0,\infty)\rightarrow[0,\infty)$, known as the settling-time function, such that the following conditions are satisfied:
        \begin{enumerate}
            \item \textit{finite-time convergence}:  For every  $x_0\in\mathcal{N}\backslash\{0\}$ and for every $t_0\in[0,\infty)$, the trajectory $\phi(x_0,t_0,t)$ defined on $[t_0,T(x_0,t_0)),~ \phi(x_0,t_0,t)\in\mathcal{N}\backslash\{0\}$ for all $t\in[t_0,T(x_0,t_0))$, and $\lim_{t\rightarrow T(x_0,t_0)}\phi(x_0,t_0,t)=0$.
            \item \textit{Lyapunov Stability}: For every $\epsilon>0$ and $t_0\in[0,\infty)$, there exists $\delta=\delta(\epsilon,t_0)>0$ such that $\mathcal{B}_{\delta}(0)\subset\mathcal{N}$ and for every $x_0\in\mathcal{B}_{\delta}(0)\backslash\{0\}, ~\phi(x_0,t_0,t)\in\mathcal{B}_{\epsilon}(0)$ for all $t\in[t_0,T(x_0,t_0)).$
        \end{enumerate}
        In addition $x(t)\equiv0$ is globally finite-time stable if $\mathcal{N}=\R^n$.
    \end{definition}
    
    The settling time function $T(x_0,t_0)$ depends on the initial condition $x_0$ of the system \eqref{sys1}. Therefore we consider a more stronger notion called fixed-time convergence \cite{TVFixed1,TVfx2}, where the settling-time function is independent of the system initialization.
    \begin{definition}(see \cite{Polyakov:FxTS})
        The zero solution $x(t)\equiv0$ of \eqref{sys1} is said to be fixed-time stable, if it is globally finite-time stable and the settling-time function is uniformly upper bounded, that is, there exists $0<T_{max}<\infty$ such that $T(x_0)\leq T_{max},\forall~ x_0\in\R^n$.
    \end{definition}
    The Lyapunov stability conditions for the fixed-time stability of system \eqref{sys1} can be established as follows:
    \begin{lemma}\cite[Lemma 1]{Polyakov:FxTS}\label{L2}
        If there exists a continuously differentiable positive definite function  $V(x,t):\D\times\R_{\geq 0}\rightarrow\R_{\geq0}$ such that 
        \begin{enumerate}
            \item $V(0,t)=0$ for all $t\geq 0$
            \item any solution $x(t)$ of \eqref{sys1} satisfies the inequality
                \begin{equation}
                    \dot{V}(x,t)\leq -\left(a V^{p}\left(x,t\right)+b V^{q}\left(x,t\right)\right)^{k}
                \end{equation}
                for some $p, q, a, b, k >0, p k<1, q k>1$, 
        \end{enumerate}
        then the origin is globally fixed-time stable for the system \eqref{sys1} and the settling-time function is
        \begin{equation}\label{T_settle_Fxts}
            T(x_0)\leq T_{\max}=\frac{1}{a(1-p)}+\frac{1}{b(q-1)}, \enspace \forall ~x_0\in\D.
        \end{equation}
    \end{lemma}
%%%%%%%%%%%%%%%%%%%%%%%%%%%%%%%%%%%%%%%%%%%%%%%%%%%%%%%%%%%%%%%%%%%%%%%%%
\section{Problem Statement}
    Consider an inequality constrained convex optimization problem with TV objective function $f:\R^n\times\R_{\geq0}\rightarrow\R$ and TV inequality constraints $g_i:\R^n\times\R_{\geq0}\rightarrow\R$:
    \begin{equation}\label{Ineq}
        \begin{aligned}
            \min_{x(t)}\enspace &f(x,t)\\
            \textrm{s.t.}\enspace & g_i(x,t)\leq 0,\enspace i=1,2,\ldots,m.
        \end{aligned}
    \end{equation}
    We assume that the objective function and the inequality constraints are continuously differentiable with respect to $t$ for all $(x,t)\in\R^n\times\R_{\geq 0}$ and the inequality constraints do not change with time at an unbounded rate, that is, $\vectornorm{\nabla_t g_i(x,t)}\leq \rho_i(t),\;\forall~t\geq 0$ where $0<\rho_i(t)<\infty$ for $i=1,2,\ldots,m$. To propose a dynamical system to track the optimal solution  of problem \eqref{Ineq}, we consider the associated Lagrangian function $L:\R^n\times\R^m\times\R_{\geq0}\rightarrow\R$ defined as,
    \begin{equation}\label{Lagrangian}
        L(x,\lambda,t)= f(x,t)+\sum_{i=1}^m \lambda_i(t)g_i(x,t)
    \end{equation}
    where $\lambda_i\in\R_{\geq0}$ is the Lagrange multiplier associated with the $i^{\textrm{th}}$ inequality constraint. 
    
    Let $x\st$ be the minimizer of the optimization problem \eqref{Ineq} at the time instant $t$ and the corresponding dual optimal argument is denoted by $\lambda\st$. Then the primal-dual optimizer trajectory formed with the optimal points from each instant of time is denoted by $(x\st(t),\lambda\st(t))$.  To provide a better clarity to the contributions of the work, we impose the following assumptions on the TV optimization problem \eqref{Ineq} throughout the paper.
    \begin{assumption}\label{Assumption_StrongConvexity}
        The objective function $f(x,t)$ is twice continuously differentiable and $\mu$-strongly convex in $x\in\R^n$ for all $t\geq0$, that is, $\xx f(x,t)\succcurlyeq \mu I_n$ for some $\mu>0$ and the inequality constraints $g_i(x,t),\; i=1,2,\ldots,m$ are convex in $x\in\R^n$ for all $t\geq 0$.
    \end{assumption}
    
    \begin{assumption}(Slater's constraint qualification)\label{Assumption_SlatersCondition}
        Let the set of feasible region at time $t$ be  $\Theta(t)=\{x(t):g_i(x,t)\leq 0\} , \enspace i=1,2,\ldots,m$ for all $t\geq 0$, then there exist an $\hat{x}(t)\in\textrm{relint}(\Theta(t))$  for all $t\geq 0$. In other words, for all $t\geq0$, there exist a point $\Hat{x}(t)$ in the relative interior of the feasible region at which all the non-affine, convex inequality constraints are strictly feasible. 
    \end{assumption}

    \begin{assumption}\cite[Assumption $S^+$]{bertsekas2014constrained} \label{Assumption_StrictComplementary}
        The dual optimizer $\lambda\st(t)$ satisfies strict complementary slackness condition at each time $t\geq0$, that is, $g_i(x\st(t),t)=0 \implies \lambda_i\st(t)>0$ for all $i=1,2,\ldots,m.$ 
    \end{assumption}
    
    \begin{assumption}\label{Assumption_linearindependence}
     Let $I(x,t)$ denotes the index set of all active constraints at the time instant $t$, that is, $I(x,t)\deff\{i\in\{1,2,\ldots,m\}:g_i(x,t)=0\}$. The cardinality of $I(x,t)$ at time $t$, denoted by $\vert I(x,t) \vert$, is less than or equal to the dimension of the primal variables at each time $t\geq 0$. Moreover, the vectors in the set $\{\x g_i(x,t):i\in I(x,t)\}$ are linearly independent for almost all $t\geq0$.
    \end{assumption}
    \begin{assumption}\label{inactiveGi}
       For every $t\geq0$, there exists at least one inequality constraint $g_i(x,t)$ for $i\in\{1,2,\ldots,m\}$  such that its gradient $\x g_i(x,t)$ is not orthogonal to $ \x f(x,t) $.
    \end{assumption}
   \begin{assumption}\label{Assumption_boundedness}
        The Hessian of the Lagrangian function is uniformly bounded for all $t\geq 0$, that is, there exists a constant $b$ with $0<b<\infty $  such that
        \begin{equation*}
            \sup_{x(t)\in\R^n}\vectornorm{\xx L(x,\lambda,t)}\le b,\enspace\forall~t\geq 0.
        \end{equation*}
        Thus $L(x,\lambda,t)$ satisfies the $\mathcal{L}_p$-Lipschitzian gradient condition
        \begin{equation*}
            \vectornorm{\x L(x_1,\lambda,t)-\x L(x_2,\lambda,t)}\leq \mathcal{L}_p\vectornorm{x_1-x_2}
        \end{equation*}
        for all $x_1,x_2\in\R^n$ and  for all $t\geq0$, where
        $\x L(x,\lambda,t)=\x f(x,t)+\sum_{i=1}^m\lambda_i(t) \x g_i(x,t).$
    \end{assumption}
    
    The following Lemma ensures the strong convexity of the Lagrangian function in $x\in\R^n$ and the proof of the Lemma is provided in Appendix \ref{Proof_Lagrangian_StrongConvexity}.
    \begin{lemma}\label{Lagragian_strongconvex}
        Under the Assumption-\ref{Assumption_StrongConvexity}, the Lagrangian function \eqref{Lagrangian} is strongly convex in $x$ for all $t\geq 0$.
    \end{lemma}
    The Assumptions \ref{Assumption_StrongConvexity} and \ref{Assumption_SlatersCondition} together ensure the  strong duality \cite{duality1} of the optimization problem \eqref{Ineq} for all $t\geq0$, and thereby the optimizer trajectory pair
    $(x\st(t),\lambda\st(t))$ must satisfy the following Karush-Kuhn-Tucker (KKT) optimality conditions  \cite{PredCorr:Mehyar} for all $t\geq0$: 
    \begin{align}
        \x f(x\st(t),t)+\sum_{i=1}^m\lambda_i\st(t)\x g_i(x\st(t),t)=0\label{KKT1}\\
        \lambda_i\st(t)g_i(x\st(t),t)=0,\;i=1,2,\ldots,m\label{KKT2}\\
        \lambda_i\st(t)\geq0,\;i=1,2,\ldots,m\label{KKT3}\\
        g_i(x\st(t),t)\leq0,\;i=1,2,\ldots,m\label{KKT4}.
    \end{align}
    The existence and uniqueness of the primal optimizer trajectory $x\st(t)$ of the optimization problem \eqref{Ineq} is guaranteed by the Assumption \ref{Assumption_StrongConvexity}. Whereas under the Assumption \ref{Assumption_linearindependence}, the KKT condition \eqref{KKT1} implies the uniqueness of the dual optimizer trajectory $\lambda\st(t)$. Since \eqref{KKT1} and \eqref{KKT2} hold for all $t\geq 0$, the time-derivative of \eqref{KKT1} and \eqref{KKT2} must also be zero and  this leads to
    \begin{align}
        &\hspace{1.8cm}\xx f(x\st(t),t)\dot{x}\st(t)+\xt f(x\st(t),t)\nonumber\\
        &+\sum_{i=1}^m\left[\dot{\lambda}\st(t)\x g_i(x\st(t),t)
        +\lambda_i\st(t)\xt g_i(x\st(t),t)\right.\nonumber\\
        &\hspace{1.8cm}+\lambda_i\st(t)\xx g_i(x\st(t),t)\dot{x}\st(t)\Big]=0\label{optimality1}\\
        &\lambda_i\st(t)\x g_i(x\st(t),t)\dot{x}\st (t)+\lambda_i\st(t)\nabla_t g_i(x\st(t),t)\nonumber\\
        &\hspace{1.8cm}+\dot{\lambda}_i\st(t)g_i(x\st(t),t)=0,\quad i=1,2,\ldots,m.\label{optimality2}
    \end{align}
    With
    \begin{align*}
        G_d &\deff \textrm{diag}\{g_1, g_2, \ldots, g_m\} \in\R^{m\times m}\\
        \nabla_t G &\deff 
            \begin{bmatrix}
                \nabla_t g_1 & \nabla_t g_2 & \ldots & \nabla_t g_m
            \end{bmatrix}^\top\in \R^m\\
        \x G&\deff 
            \begin{bmatrix}
                \x g_1 & \x g_2 & \cdots & \x g_m
            \end{bmatrix}\in R^{n\times m}\\
        \lambda(t)\circ \x G^\top &\deff
            \begin{bmatrix}
                \lambda_1\x g_1^\top & \lambda_2\x g_2^\top & \ldots &\lambda_m\x g_m^\top
            \end{bmatrix}^\top
    \end{align*}
    equations \eqref{optimality1} and \eqref{optimality2} can be together rewritten as
    \begin{align}\label{optimality3}
        J(x\st,\lambda\st,t)        
        \begin{bmatrix}
            \dot{x}\st(t)\\\dot{\lambda}\st(t)
        \end{bmatrix}+
        H(x\st,\lambda\st,t)=0,
    \end{align}
    where 
    \begin{align*}
        J(x\st,\lambda\st,t)&= 
        \begin{bmatrix}
            \xx L(x\st,\lambda\st,t) & \x G(x\st,t)\\
            \lambda\st(t)\circ\x G^\top(x\st,t) & G_d(x\st,t)
        \end{bmatrix}\\
        H(x\st,\lambda\st,t) &=
        \begin{bmatrix}
            \xt f(x\st,t)+\sum_{i=1}^m\lambda_i\st(t)\xt g_i(x\st,t)\\
            \lambda\st(t)\circ\nabla_t G(x\st,t)
        \end{bmatrix}\\
        \xx L(x,\lambda,t) &= \xx f(x,t)+\sum_{i=1}^m\lambda_i(t)\xx g_i(x,t).
    \end{align*}
    Due to space constraints we use $(x\st,\lambda\st)$ instead of $(x\st(t),\lambda\st(t))$ for evaluating the components of the $J(x\st(t),\lambda\st(t),t)$ and $H(x\st(t),\lambda\st(t),t)$ matrices.
    
    Since $\xx L$ is invertible under the Assumption \ref{Assumption_StrongConvexity}, the Schur complement of the block  $\xx L$ of the matrix $J(x\st,\lambda\st,t)$ is $$M(x\st(t),\lambda\st(t),t)=G_d-\lambda\st(t) \circ \underbrace{\x G^\top\xx L^{-1}\x G}_{P(x\st(t),\lambda\st(t),t)},$$ where $P(x\st(t),\lambda\st(t),t)$ is a positive definite matrix under Assumption \ref{Assumption_linearindependence}. The positive definiteness of $P(x\st(t),\lambda\st(t),t)$ together with the Assumption \ref{Assumption_StrictComplementary}, ensures that the matrix $M(x\st(t),\lambda\st(t),t)$ is negative definite for all $t\geq0$. 
    
   \begin{lemma}\label{Lemma_Invertibility}
        Under the Assumptions \ref{Assumption_StrongConvexity}, \ref{Assumption_StrictComplementary} and \ref{Assumption_linearindependence} the matrix $J(x\st(t),\lambda\st(t),t)$ is invertible at all $t\geq 0$.
    \end{lemma}
    
    The proof of Lemma \ref{Lemma_Invertibility} is provided in the Appendix \ref{Proof_Lemma_Invertibility}. \\Solving for $x\st(t)$ and $\lambda\st(t)$ from \eqref{optimality3} yields the following dynamical system:
    \begin{equation}\label{OptimalityCondition}
        \begin{bmatrix}
            \dot{x}\st(t)\\\dot{\lambda}\st(t)
        \end{bmatrix} = -J^{-1}(x\st(t),\lambda\st(t),t)H(x\st(t),\lambda\st(t),t) 
    \end{equation}
    that characterizes the primal and dual optimizer trajectory $(x\st(t),\lambda\st(t))$ of the problem \eqref{Ineq}.
    The following Lemma ensures that the gradient of the Lagrangian function with respect to $x$  vanishes only along the unique primal-dual optimizer trajectory of problem \eqref{Ineq}.
    \begin{lemma}\label{Lemma_StrongConvexity}
        Let $\mathcal{Z}(t)=\{(x(t),\lambda(t)): \x L(x,\lambda,t)=0\}$. Then $(x(t),\lambda(t))\in \mathcal{Z}(t) $ if and only if $(x(t),\lambda(t))=(x\st(t),\lambda\st(t))$ for all $t\geq 0$. Moreover, under the Assumptions \ref{Assumption_StrongConvexity}, \ref{Assumption_SlatersCondition} and \ref{Assumption_linearindependence}, the optimizer trajectory $(x\st(t),\lambda\st(t))$ of the TV optimization problem is unique for all $t\geq 0$.
    \end{lemma}
    The proof of Lemma \ref{Lemma_StrongConvexity} is given in Appendix \ref{Proof_Lemma_StrongConvexity}.
 
    In the next section we propose a projected dynamical system to track the optimizer trajectory of the TV inequality constrained optimization problem \eqref{Ineq} in a way that, the proposed system reduces to the optimizer trajectory characterization \eqref{optimality3} along the optimizer trajectory $(x\st(t),\lambda\st(t))$. 
%%%%%%%%%%%%%%%%%%%%%%%%%%%%%%%%%%%%%%%%%%%%%%%%%%%%%%%%%%%%%%%%%%%%%%%%%
\section{Asymptotic Convergence of Inequality Constrained TV convex optimization problems}   \label{section4}

    In this section, we follow the steps below to analyze the convergence of the trajectories of a proposed dynamical system to the optimizer trajectory $\left(x\st(t),\lambda\st(t)\right) $ of the TV inequality convex optimization problem \eqref{Ineq}:
\begin{itemize}
    \item Construct a dynamical system to track the optimal solution of problem \eqref{Ineq}
    \item  Formulate the proposed dynamical system as a projected dynamical system and establish the existence of its solution
    \item Demonstrate the convergence of the trajectories of the proposed projected dynamical system using a Lyapunov-based stability analysis.
\end{itemize}
    
    To propose a dynamical system to track the optimal solution of the problem \eqref{Ineq}, we revise the optimizer trajectory characterization \eqref{optimality3} as follows:
    
    \begin{equation}\label{JSystem}
        J(x,\lambda,t)
        \begin{bmatrix}
            \dot{x}(t)\\\dot{\lambda}(t)
        \end{bmatrix}
        =H_{pred}+H_{corr}+\tilde{J}H_{aug},
    \end{equation}
where 
\begin{align*}
    J&=
    \begin{bmatrix}
        \xx L(x,\lambda,t) & \x G(x,t)\\ \lambda(t)\circ\x G^\top(x,t) & G_d(x,t)
    \end{bmatrix}\\
    H_{pred}&=
    \begin{bmatrix}
        -\xt f(x,t)-\xt G(x,t)\lambda(t)\\
        -\lambda\circ \nabla_t G(x,t)
    \end{bmatrix}\\
    H_{corr}&=
    \begin{bmatrix}
        -\alpha\x L(x,\lambda,t)\\
        G_d(x,t)\lambda(t)
    \end{bmatrix},\quad \alpha>0\\
    H_{aug}&=
    \begin{bmatrix}
        -\xx L^{-1}\x G\x G^\top\x L\\
        \x G^\top\x L
    \end{bmatrix}.
\end{align*}
    The motivation of introducing the vectors in the right-hand side of \eqref{JSystem}
    and the matrix $\tilde{J}$ will be discussed later in this section.
    We first analyse the invertibility of the matrix $J(x,\lambda,t)$. The matrix is invertible at each time $t$, if and only if the Schur complement of the invertible block $\xx L(x,\lambda,t)$  of $J(x,\lambda,t)$,
    \begin{equation}\label{Schur}
        M(x,\lambda,t)=G_d-\lambda\circ\x G^\top\xx L^{-1}\x G
    \end{equation}
    is invertible.
    The matrix $M$ becomes singular if there exists an $i\in\{1,2,\ldots,m\}$ for which either $\lambda_i(t)=g_i(x,t)=0$ or $g_i=\lambda_i\x g_i^\top\xx L^{-1}\x g_i$ with zero off-diagonal entries at some time $t_1>0$.
    Therefore to ensure the invertibility of $M$ at all time $t\geq 0$, we approximate the Schur matrix denoted by $\Tilde{M}$ as,
    \begin{equation}
        \Tilde{M}(x,\lambda,t)=M(x,\lambda,t)-S(t)
    \end{equation}
    where $S(t)$ is a diagonal matrix, called the slack matrix with $i^{\textrm{th}}$ diagonal entry $s_i(t)$. 
    Each slack variable $s_i(t)$ is chosen  such that $s_i(t)>0$ for all $t\geq0$ and $s_i(t)\rightarrow0$ as $t\rightarrow\infty$, as a result  $\tilde{M}(x,\lambda,t)\rightarrow M(x,\lambda,t)$ as $t\rightarrow\infty$.
    The approximated inverse $\tilde{J}^{-1}(x,\lambda,t)$ is calculated by replacing the Schur matrix $M$ with $\Tilde{M}$ in the computation of the inverse of $J(x,\lambda,t)$ and then the proposed system \eqref{JSystem} can be rewritten as,
    \begin{equation}\label{JSystemApprox}
        \begin{bmatrix}
            \dot{x}(t)\\\dot{\lambda}(t)
        \end{bmatrix}=\tilde{J}^{-1}\left(H_{pred}+H_{corr}\right)+H_{aug}.
    \end{equation}
    We define a vector field $\X:\R^n\times\R^m\times\R_{\geq0}\rightarrow\R^n\times\R^m$ to denote the right-hand side of \eqref{JSystemApprox} as
    \begin{equation*}
            \X(x,\lambda,t) = \begin{bmatrix} \X_x(x,\lambda,t)\\\X_{\lambda}(x,\lambda,t)\end{bmatrix}
        \end{equation*}
    where $\mathcal{X}_x:\R^n\times\R^m\times\R_{\geq0}\rightarrow\R^n$ and $\mathcal{X}_{\lambda}:\R^n\times\R^m\times\R_{\geq0}\rightarrow\R^m$ are given in \eqref{X_Dynamics} and \eqref{Lambda_Dynamics}.
    Under the assumption of strong duality, the KKT conditions \eqref{KKT1} to \eqref{KKT4} become the necessary and sufficient conditions for optimality. Therefore to ensure the dual feasibility condition \eqref{KKT3}, we incorporate a component-wise projection in the $\lambda(t)$ dynamics  and propose a continuous-time dynamical system
    \begin{equation}\label{Dyn}
        \begin{aligned}
            \dot{x}(t) &= \mathcal{X}_{x}(x,\lambda,t),\quad &x(0)=x_0\\
            \dot{\lambda}(t) &= \left[\mathcal{X}_{\lambda} (x,\lambda,t)\right]_\lambda^+,\quad&\lambda(0)=\lambda_0.
        \end{aligned}
    \end{equation}
    
     \begin{figure*}
        \begin{align}
        \X_x =& -\xx L^{-1}(x,\lambda,t)\Bigg[\alpha\x L(x,\lambda,t)+\xt f(x,t)+\xt G(x,t)\lambda(t)+\x G(x,t)\x G^\top(x,t)\x L(x,\lambda,t)\nonumber\\
        &+\x G(x,t)\Tilde{M}^{-1}(x,\lambda,t)\left(\lambda(t)\circ\x G^\top(x,t)\right)\xx L^{-1}(x,\lambda,t)\left(\alpha\x L(x,\lambda,t)+\xt f(x,t)+\xt G(x,t)\lambda(t)\right)\nonumber\\
        &-\x G(x,t)\Tilde{M}^{-1}(x,\lambda,t)\Big(\lambda\circ 
        \nabla_t G(x,t)-G_d(x,t)\lambda(t)\Big)\Bigg]\label{X_Dynamics}\\
        \X_{\lambda}=& \Tilde{M}^{-1}(x,\lambda,t)\lambda(t)\circ\left(\x G^\top(x,t)\xx L^{-1}(x,\lambda,t)\left(\alpha\x L(x,\lambda,t)+\xt f(x,t)+\xt G(x,t)\lambda(t)\right)-\nabla_t G(x,t)\right)\nonumber\\
        &+\Tilde{M}^{-1}(x,\lambda,t)G_d(x,t)\lambda(t)+\x G^\top(x,t)\x L(x,\lambda,t)\label{Lambda_Dynamics}.
    \end{align}
    \end{figure*}  

    The proposed dynamical system mainly comprises of three parts. Besides the prediction part $H_{pred}$ (derived from the optimizer trajectory characterization \eqref{OptimalityCondition}) and correction parts $H_{corr}$ (arising from the KKT optimality conditions \eqref{KKT1} and \eqref{KKT2}), it consists of an additional term $H_{aug}(x,\lambda,t)$. 
    
     In the absence of this augmented term in \eqref{Lambda_Dynamics}, if there exists a $\Bar{t}\geq 0$ such that $\lambda(\Bar{t})=0$, then $\X_{\lambda}=0$ for all $t\geq\Bar{t}$ and it forces each trajectory $\lambda_i(t)$ to stay at zero for all $t\geq \Bar{t}$. This impedes the convergence of the trajectories of the dynamical system \eqref{Dyn} to the optimizer trajectory $(x\st(t),\lambda\st(t))$ unless $\lambda(t)=\lambda\st(t)=0$ for all $t\geq \Bar{t}$. Thus the utility of the augmented term is to preclude the possibility of the convergence of the trajectories of \eqref{Dyn} to a non-optimizer trajectory. The following Lemma \ref{lemmaAdditionalTerm} ensures that the augmented term vanishes only along the optimizer trajectory.
    
    \begin{lemma}\label{lemmaAdditionalTerm}
        Along the trajectories of the dynamical system \eqref{Dyn}, the augmented term $ \x G^\top(x,t)\x L(x,\lambda,t)\equiv0$ if and only if $(x(t),\lambda(t))=(x\st(t),\lambda\st(t))$ for all $t\geq 0$.
    \end{lemma}
    
    The proof of Lemma \ref{lemmaAdditionalTerm} is given in Appendix \ref{lemmaAdditionalTerm_Proof}.
    Thus the term $H_{aug}(x,\lambda,t)$ guarantees the dynamical system \eqref{Dyn} to track the optimizer trajectories by allowing the Lagrange multiplier to shift from a non-optimal zero value to a positive value. As a result, the proposed dynamical system can track the optimizer trajectories during the switching of the inequality constraints from active to inactive state and vice-versa.
    
    We denote the right-hand side of \eqref{Dyn} by $\X_{pd}:\R^n\times\R^m_{\geq 0}\times\R_{\geq0}\rightarrow\R^n\times\R^m$. 
    The vector field $\X_{pd}$ is discontinuous on the set $\Xi_i(t)\deff\{(x(t),\lambda_i(t))\in\R^n\times\R_{\geq0}:\lambda_i(t)=0\textrm{ and } \left(\X_{\lambda}\right)_i<0\}$ for $i=1,2,\ldots,m$ and $t\geq 0$. 
    To guarantee the existence of a solution of \eqref{Dyn}, the following Lemma expresses the dynamical system \eqref{Dyn} as a projected dynamical system \cite{nagurney1995projected}.
    \begin{lemma}\label{ProjectedSystem}
        The proposed dynamical system $\X_{pd}$ can be represented as a projected dynamical system.
    \end{lemma}
    The proof of Lemma \ref{ProjectedSystem} is provided in Appendix \ref{Lemma_ProjectedSystem}.
    We consider the solution notion of \eqref{Dyn} in the Carath{\'e}odary sense \cite{NonsmoothCortes}. A map $\phi:[0,T)\rightarrow\R^n\times\R^m$ is a Carath{\'e}odary solution of $\X_{pd}$ on the interval $[0,T)$ if it is absolutely continuous on $[0,T)$ and satisfies $\dot{\phi}(t)=\X(\phi(t))$ almost everywhere in $[0,T)$. Since we assign a particular value to the vector field on the discontinuous set $\Xi_i(t)$ to achieve the dual feasibility, the corresponding set-valued map obtained through the regularization \cite{HALKrasovskiiCaratheodary} of \eqref{Dyn} effectively reduces to a map with a singleton set. As a result, the Carath{\'e}odary solution of \eqref{Dyn} becomes same as the Krasovskii solution \cite{SHu} of the corresponding set-valued map and thus the existence of a Carath{\'e}odary solution to the dynamical system \eqref{Dyn} is guaranteed by \cite[Theorem 2]{projectedPowerSystem} for all $t\geq 0$.

    In the following subsection, we uncover the main contribution of the paper, which establishes the asymptotic convergence of the trajectories of the proposed dynamical system \eqref{Dyn} to the optimizer trajectory the problem \eqref{Ineq}.
    \subsection{Stability Analysis}

\begin{theorem}\label{TheoremEXponentialStability}
        Under the Assumptions \ref{Assumption_StrongConvexity}, \ref{Assumption_SlatersCondition}, \ref{Assumption_linearindependence} and \ref{inactiveGi}, the solution $(x\st(t),\lambda{\st(t)})$ of the dynamical system \eqref{Dyn} is locally asymptotically stable.
    \end{theorem}
    \begin{proof}
        Consider the following candidate Lyapunov function
        $V:\R^n\times\R^m_{\geq0}\times\R_{\geq0}\rightarrow\R$,
        \begin{align}\label{Lyap1} 
          V(x,\lambda,t)= \frac{1}{2}\vectornorm{\x L(x,\lambda,t)}^2
        \end{align}
        with 
        $V(x,\lambda,t)=0$ if and only if $(x(t),\lambda(t))=(x\st(t),\lambda\st(t))$ by Lemma \ref{Lemma_StrongConvexity}. 
        Then the derivative of $V(x,\lambda,t)$ with respect to time along the trajectories of the dynamical system \eqref{Dyn} is,
        \begin{align*}
            \dot{V} =& \x L(x,\lambda,t)^\top\Bigg[\xx L(x,\lambda,t)\dot{x}(t)+\xt f(x,t)\\
            &+\xt G(x,t) \lambda(t)+\x G(x,t)\dot{\lambda}(t)\Bigg]\\
            =& \x L(x,\lambda,t)^\top\Bigg[\xx L(x,\lambda,t)\dot{x}(t)+\xt f(x,t)\\
            &+\xt G(x,t)\lambda(t)+\x G(x,t)\mathcal{X}_{\lambda}\\
            &+\x G(x,t)\left(\left[\mathcal{X}_{\lambda}\right]_{\lambda}^+-\mathcal{X}_{\lambda}\right) \Bigg].
        \end{align*}
        Substituting \eqref{X_Dynamics} and \eqref{Lambda_Dynamics} in $\dot{V}$ results in
        \begin{align}
            \dot{V} &=\x L^\top\left[-\alpha\x L+\x G(x,t)\big(\left[\mathcal{X}_{\lambda}\right]_{\lambda}^+-\mathcal{X}_{\lambda}\big)\right].\label{Vdot}
        \end{align}
        We evaluate the $i^{th}$ element of the term $W(t)\deff\x G(x,t)\big(\left[\mathcal{X}_{\lambda}\right]_{\lambda}^+-\mathcal{X}_{\lambda}\big)$ in three different cases as follows:
        \begin{enumerate}[i.]
            \item $\lambda_i>0$, which implies $\left[\left(\mathcal{X}_{\lambda}\right)_i\right]_{\lambda_i}^+=\left(\mathcal{X}_\lambda\right)_i$ and $W_i(t)=0$.
            
            \item $\lambda_i=0$ and $\max\{0,\left(\mathcal{X}_\lambda\right)_i\}=\left(\mathcal{X}_\lambda\right)_i$, then $\left[\left(\mathcal{X}_{\lambda}\right)_i\right]_{\lambda_i}^+=\left(\mathcal{X}_\lambda\right)_i$ and $W_i(t)=0$.
            
            \item $\lambda_i=0$ and $\max\{0,\left(\mathcal{X}_\lambda\right)_i\}=0$, then $\left[\left(\mathcal{X}_\lambda\right)_i\right]_{\lambda_i}^+=0$ and 
                \begin{equation*}
                    W(t) = -\x G(x,t)\x G^\top(x,t)\x L(x,\lambda,t).
                \end{equation*}
        \end{enumerate}
        Then \eqref{Vdot} becomes,
        \begin{align}
            \dot{V}&=-\alpha\x L^\top(x,\lambda,t)\x L(x,\lambda,t)\nonumber\\
            &-\x L^\top(x,\lambda,t)\x G(x,t)\x G^\top(x,t)\x L(x,\lambda,t)\nonumber\\
            &=-\alpha\vectornorm{\x L(x,\lambda,t)}^2-\vectornorm{\x G^\top(x,t)\x L(x,\lambda,t)}^2\nonumber\\
            &\leq-\alpha\vectornorm{\x L(x,\lambda,t)}^2\nonumber\\
            \dot{V}&\leq -2\alpha V.\label{vdotfinal}
        \end{align}
        Then by \cite[Theorem 3.1]{khalil2002nonlinear} the set $\mathcal{Z}(t)$ is locally asymptotically stable. 
        By Lemma \eqref{Lemma_StrongConvexity}, it leads to the conclusion that the trajectories of the proposed dynamical system \eqref{Dyn} asymptotically converge to the unique optimizer trajectory $(x\st(t),\lambda\st(t))$ of the TV inequality constrained convex optimization problem \eqref{Ineq}.
    \end{proof}
%%%%%%%%%%%%%%%%%%%%%%%%%%%%%%%%%%%%%%%%%%%%%%%%%%%%%%%%%%%%%%%%%%%%%%%%%    
\section{Fixed-time convergence of Inequality constrained TV convex optimization problems} 
    The dynamical system proposed in the previous section can track the optimizer trajectory of the underlying optimization problem in an asymptotic sense, that is, the trajectories of the proposed system \eqref{Dyn} converge to the optimizer trajectory as $t\rightarrow \infty$. Therefore to track the optimizer trajectory of the optimization problem \eqref{Ineq} in a fixed time, we modify the correction term $H_{corr}$ in \eqref{JSystem} as follows:
    \begin{equation}
        \tilde{H}_{corr} = 
            \begin{bmatrix}
                \textcolor{blue}{-}\dfrac{c_1\x L(x,\lambda,t)}{\vectornorm{\x L(x,\lambda,t)}^{\gamma_1}}\textcolor{blue}{-}\dfrac{c_2\x L(x,\lambda,t)}{\vectornorm{\x L(x,\lambda,t)}^{\gamma_2}}\\
                G_d(x,t)\lambda(t)
            \end{bmatrix}
    \end{equation}
    where $c_1,c_2\in\R_{>0}$, $\gamma_1\in(0,1)$, $\gamma_2<0$ and we define $\dfrac{\x L(x,\lambda,t)}{\vectornorm{\x L(x,\lambda,t)}^{\gamma_1}}=0$ if $\x L(x,\lambda,t)=0$ for all $t\geq 0$. 
    This modification in the correction term results in the following dynamical system:
    \begin{equation}\label{FxtsDyn}
        \begin{aligned}
            \dot{x}(t) &= \tilde{\mathcal{X}}_x(x,\lambda,t),\quad &x(0)=x_0\\
            \dot{\lambda}(t) &= \left[\tilde{\mathcal{X}}_{\lambda}(x,\lambda,t)\right]_{\lambda}^+,\quad &\lambda(0)=\lambda_0
        \end{aligned}
    \end{equation}
    where $\tilde{\mathcal{X}}_x:\R^n\times\R^m\times\R_{\geq0}\rightarrow\R^n$ and $\tilde{\mathcal{X}}_{\lambda}:\R^n\times\R^m\times\R_{\geq0}\rightarrow\R^m$ are defined as 
    \begin{equation}\label{FXtsVectorField}
        \begin{bmatrix}
            \tilde{\mathcal{X}}_x\\\tilde{\mathcal{X}}_{\lambda}
        \end{bmatrix}=
        \tilde{J}^{-1}\left(H_{pred}+\tilde{H}_{corr}\right)+H_{aug}.
    \end{equation}
    The dynamical system \eqref{FxtsDyn} is a nonsmooth  due to the projection operator which enforces the  nonnegativity of the Lagrange multipliers. However the vector fields $\tilde{X}_x$ and $\tilde{\mathcal{X}}_{\lambda}$ is continuous for all $(x(t),\lambda(t))\in\R^n\times\R^m$, as guaranteed by the following Lemma.
    \begin{lemma}\label{Lemma_ContinuityofFXTS}
        Under the Assumptions \ref{Assumption_StrongConvexity}, \ref{Assumption_SlatersCondition} and \ref{Assumption_linearindependence}-\ref{Assumption_boundedness}, the right hand side of \eqref{FXtsVectorField} is continuous for all $x(t)\in\R^n$ and $\lambda(t)\in\R^m$.
    \end{lemma}
    The proof of Lemma \ref{Lemma_ContinuityofFXTS} is provided in Appendix \eqref{Proof_Lemma_Continuity}. Thus the Lemma \ref{Lemma_ContinuityofFXTS} ensures that the discontinuity in the proposed dynamical system is only due to the projection operator in the $\lambda(t)$ dynamics. Therefore the existence of the Carath{\'e}odary solutions to the dynamical system \eqref{FxtsDyn} is ensured with the similar arguments from the previous section.\\
    The following Proposition guarantees the convergence of the trajectories of \eqref{FxtsDyn} to the optimizer trajectory of the TV optimization problem \eqref{Ineq} in fixed-time.
    \begin{proposition}\label{Proposition_FXTs}
        Under the Assumptions \ref{Assumption_StrongConvexity}, \ref{Assumption_SlatersCondition} and \ref{Assumption_linearindependence}-\ref{Assumption_boundedness}, the set $\mathcal{Z}(t)$  of the projected dynamical system \eqref{FxtsDyn} is fixed-time stable and its trajectories converge to the primal-dual optimizer trajectory $(x\st(t),\lambda\st(t))$ of the optimization problem \eqref{Ineq} within a fixed time for any initial conditions $x(0)\in\R^n$ and $\lambda(0)\in\R^m_{\geq0}$. Moreover, the settling time function of \eqref{FxtsDyn} is upper bounded by $T_s\leq T_{\max}=\frac{1}{c_1\gamma_1}2^{\frac{\gamma_1}{2}}-\frac{1}{c_2\gamma_2}2^{\frac{\gamma_2}{2}}$.
    \end{proposition}
    \begin{proof}
        To proceed with the proof, we consider the following candidate Lyapunov function
        $V_1:\R^n\times\R^m_{\geq0}\times R_{\geq0}\rightarrow\R_{\geq0}$:
        \begin{align}\label{LyapunovFunction}
            V_1 = \dfrac{1}{2}\vectornorm{\x L(x,\lambda,t)}^2
        \end{align}
        with $V_1(x,\lambda,t)=0$ if and only if $(x(t),\lambda(t))=(x\st(t),\lambda\st(t))$ by Lemma \ref{Lemma_StrongConvexity}.
        The derivative of $V_1(x,\lambda,t)$ with respect to time along the trajectories of  \eqref{FxtsDyn} is,
        \begin{align*}
            \dot{V}_1 =& \x L(x,\lambda,t)^\top\Bigg[\xx L(x,\lambda,t)\dot{x}(t)+\xt f(x,t) \\
            &+\xt G(x,t)\lambda(t)+\x G(x,t)\dot{\lambda}(t)\Bigg]\\
            =& \x L(x,\lambda,t)^\top\Bigg[\xx L(x,\lambda,t)\dot{x}(t)+\xt f(x,t) \\
            &+\xt G_(x,t)\lambda(t)+\x G(x,t)\mathcal{X}_{\lambda}\\
            &+\x G(x,t)\left(\left[\mathcal{X}_{\lambda}\right]_{\lambda}^+-\mathcal{X}_{\lambda}\right)\Bigg].
        \end{align*}
        On substituting the vector fields $\tilde{\mathcal{X}}_x$ and $\tilde{\mathcal{X}}_{\lambda}$, we get
        \begin{align*}
            \dot{V}_1 &=\x L^\top(x,\lambda,t)\bigg[-\sum_{i=1}^2\dfrac{c_i\x L(x,\lambda,t)}{\vectornorm{\x L(x,\lambda,t)}^{\gamma_i}}\\
            &+\x G(x,t)\left(\left[\mathcal{X}_{\lambda}\right]_{\lambda}^+-\mathcal{X}_{\lambda}\right)\bigg]\\
            &\leq\x L^\top(x,\lambda,t)\bigg[-\sum_{i=1}^2\dfrac{c_i\x L(x,\lambda,t)}{\vectornorm{\x L(x,\lambda,t)}^{\gamma_i}}\\
            &-\x L^\top(x,\lambda,t)\x G(x,t)\x G^\top(x,t)\x L(x,\lambda,t)\bigg]\\
            &=-\sum_{i=1}^2c_i\vectornorm{\x L(x,\lambda,t)}^{2-\gamma_i}\\
            &-\vectornorm{\x G^\top(x,t)\x L(x,\lambda,t)}^2\\
            &\leq -c_1\vectornorm{\x L(x,\lambda,t)}^{2-\gamma_1}-c_2\vectornorm{\x L(x,\lambda,t)}^{2-\gamma_2}\\
            &=-\left(\alpha_1V_1^{1-\frac{\gamma_1}{2}}+\alpha_2 V_1^{1-\frac{\gamma_2}{2}}\right)^k
        \end{align*}
        where $\alpha_1\deff c_12^{1-\frac{\gamma_1}{2}}$, $\alpha_2\deff c_22^{1-\frac{\gamma_2}{2}}$ and $k=1$. Then by \cite[Lemma 1]{Polyakov:FxTS}, $\mathcal{Z}(t)$ of the system \eqref{FxtsDyn} is fixed-time stable. Hence by Lemma \ref{Lemma_StrongConvexity}, the trajectories of the proposed dynamical system \eqref{FxtsDyn} converges to the unique primal-dual optimizer trajectory of  problem \eqref{Ineq} in fixed-time irrespective of the system initialization with a settling-time function $T_s$ upper bounded by $T_{\max}=\frac{1}{c_1\gamma_1}2^{\frac{\gamma_1}{2}}-\frac{1}{c_2\gamma_2}2^{\frac{\gamma_2}{2}}$. 
    \end{proof}
    \begin{remark}
        To achieve the convergence to the optimizer trajectories in fixed-time, the slack matrix is chosen such that $S(t)\rightarrow0$ as $t\rightarrow T_{\max}$. In other words, $s_i(t)>0$ for all $t\geq 0$ and $s_i(t)=0$ for all $t\geq T_{max}$. 
    \end{remark}
    The following Corollary establishes the conditions for achieving finite-time stability of the proposed dynamical system \eqref{FxtsDyn}.
    \begin{corollary}[Finite time Convergence]
        If the parameter $c_2=0$ in the proposed dynamics \eqref{FxtsDyn}, then the set $\mathcal{Z}(t)$ is finite-time stable and it guarantees the finite-time convergence of the system trajectories to the optimizer trajectory of the inequality constrained TV optimization problem \eqref{Ineq} depending on the system initialization. Moreover it results in a settling-time function $T(x_0,\lambda_0)\leq\frac{1}{c_1\gamma_1}\vectornorm{\x L(x_0,\lambda_0,t_0)}^{\gamma_1}$.
    \end{corollary}
    The Corollary can be proved using the Lyapunov based analysis for finite-time stability of time-varying systems given in \cite[Theorem 4.1]{Haddad:FTS_ACC} with the Lyapunov function \eqref{LyapunovFunction}.

     To validate the prediction-correction based projected primal-dual dynamical systems presented in this work and to establish a simulation based comparison with the existing prediction correction dynamical system from \cite{PredCorr:Mehyar}, we consider an example of a TV inequality constrained convex optimization problem in the next subsection.

    \subsection{Numerical Example}
     Consider the following inequality constrained TV convex optimization problem:
     \begin{equation}\label{NumEx}
         \begin{aligned}
             \min_x\enspace &\frac{1}{2}\left(x_1+\sin(t)\right)^2+\frac{3}{2}\left(x_2+\cos(t)\right)^2\\
             \textrm{s.t.} \enspace & x_2-x_1-\cos(t) \leq0,
         \end{aligned}
     \end{equation}
     \begin{figure}
        \centering
        \includegraphics[width=\linewidth]{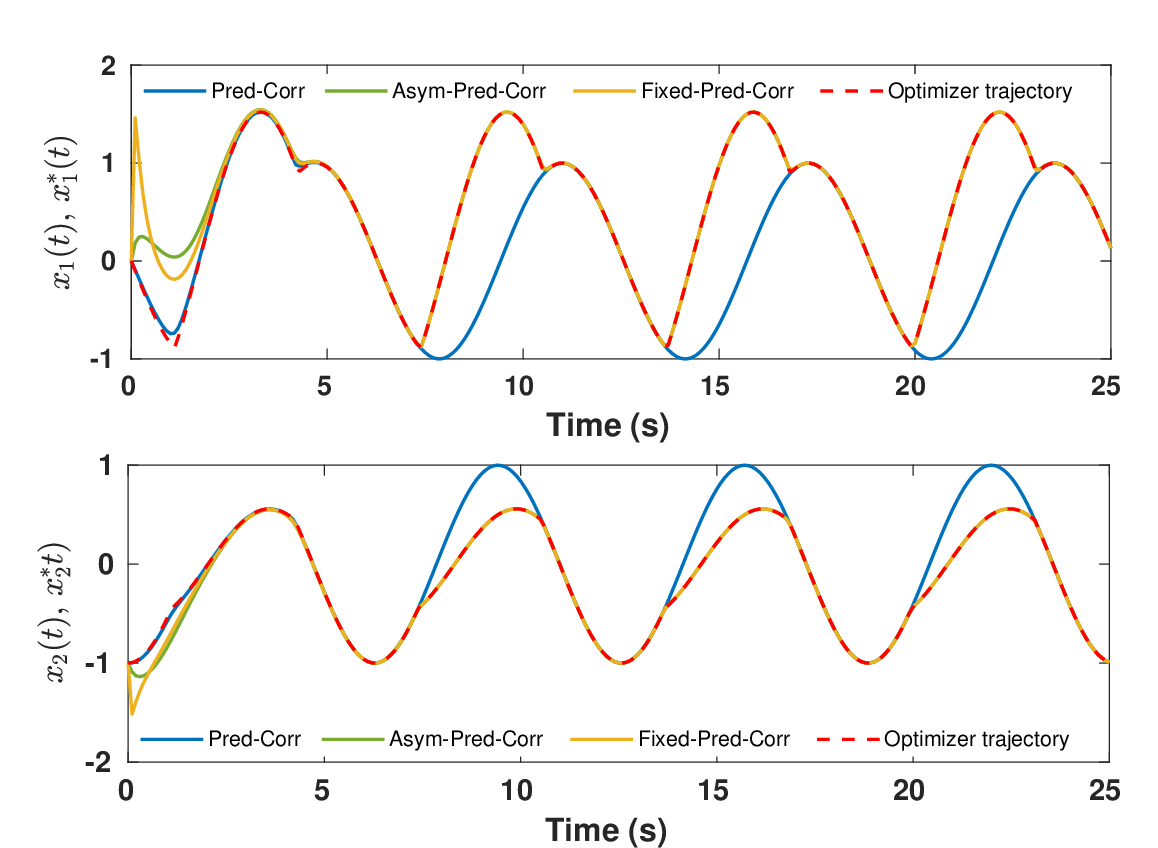}
        \caption{Convergence of system trajectory of various dynamical systems to the optimizer trajectory of \eqref{NumEx}.}
        \label{NumFig}
    \end{figure}
     with $x=\left(x_1,x_2\right)\in\R^2$. To track the optimal trajectory of the problem \eqref{NumEx} and to make a comparative study with the existing dynamical system approach, we implement the proposed dynamical systems \eqref{Dyn}, \eqref{FxtsDyn} and the prediction correction dynamics from \cite{PredCorr:Mehyar} with initial condition $x(0)=(2,1)$ and $\lambda(0)=4$ using \texttt{ode} solvers in \texttt{MATLAB.} The dynamical systems presented in this work are implemented with parameters $\alpha=2$, $c_1=c_2=1$, $\gamma_1=0.2$, $\gamma_2=-2$ and slack variable $s(t)=\left(\vectornorm{g(x)}+0.01\right)e^{-t}$, whereas the dynamical system from \cite{PredCorr:Mehyar} is implemented with the same parameters as given in their work. The coordinate wise convergence of the system trajectory $x(t)$ of various dynamical systems towards the optimizer trajectory $x\st(t)$, obtained through batch process using the \texttt{CVX} toolbox, is depicted in Figure \ref{NumFig}. The trajectories of `Pred-Corr' in Figure \ref{NumFig} corresponds to the trajectories generated with the prediction-correction dynamical system presented in \cite{PredCorr:Mehyar}, `Asym-Pred-Corr' and `Fixed-Pred-Corr' are the trajectories of asymptotic prediction-correction dynamical system \eqref{Dyn} and the fixed time prediction-correction dynamical system \eqref{FxtsDyn}. The results demonstrate that the trajectories of the asymptotic and fixed time projected dynamical systems presented in this work, accurately track the optimizer trajectory after an initial transient phase of $3$ seconds. However, the prediction-correction dynamical system from \cite{PredCorr:Mehyar} fails to adequately track the sharp changes in the optimizer trajectory, resulting in periodic intervals where the tracking is unsuccessful. This example clearly highlights the advantages of the proposed projected dynamical system over existing approaches for solving a TV inequality constrained convex optimization problems. 

    In the next section, as an application of the presented approach we formulate the TV counterpart of the extended Fermat-Torricelli problem introduced in \cite{rejitha} and tracks the solution of the illustrative examples with the asymptotic and fixed-time dynamical approaches presented in this work.
%%%%%%%%%%%%%%%%%%%%%%%%%%%%%%%%%%%%%%%%%%%%%%%%%%%%%%%%%%%%%%%%%%%%%%%%%
\section{Time-Varying extended Fermat-Torricelli Problem: Tracking of the optimizer trajectory}
\begin{figure*}[t]
        \centering
        \includegraphics[width=\linewidth]{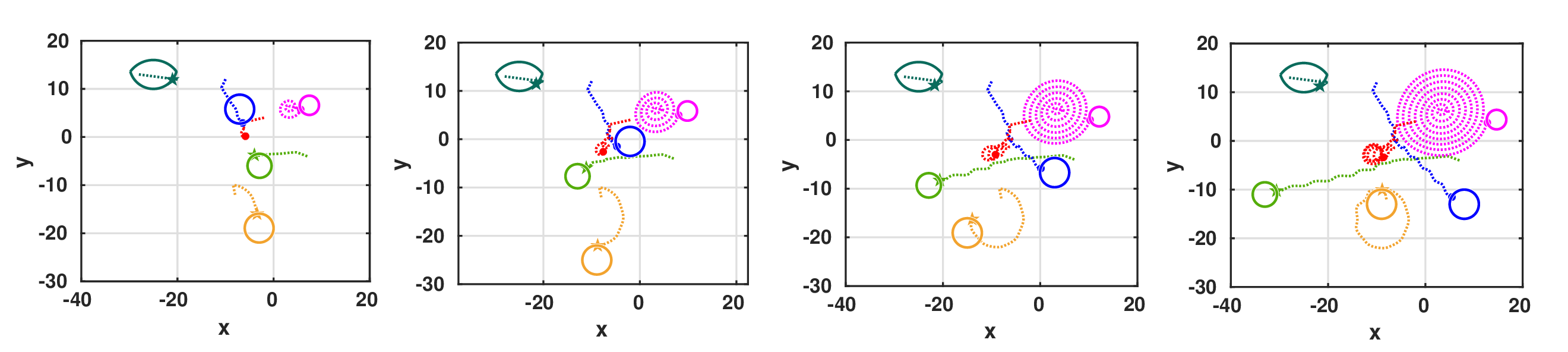}
        \caption{Example 1: Evolution of the trajectories of the fixed-time dynamical system \eqref{FxtsDyn} with with time. The red dotted line indicates the optimizer trajectory $x\st(t)$ that minimizes the sum-of-squared distances. Subplots corresponds to snapshots at $t=T_f/4,T_f/2,3T_f/4$ and $T_f$ with $T_f=50$ seconds.}
        \label{fig1_exp}
    \end{figure*}
   An eFTP is an optimization problem that solves for a point/set of points that minimizes the sum-of-distances to a finite number of time-invariant, nonempty, closed and convex sets in an $n$-dimensional Euclidean space with at least one bounded convex set. It also solves the points on each convex set, which contributes to the minimum sum-of-distances. The eFTP is formulated in \cite{rejitha} as follows:
   \begin{equation}\label{ftp_timeInvariant}
        \begin{aligned}
            \min_{X,X_i}\enspace&\sum_{i=1}^m\vectornorm{X-X_i}\\
            \textrm{s.t.}\enspace& X_i\in\Omega_i,\quad i=1,2,\ldots,q
        \end{aligned}
    \end{equation}
    where $\Omega_i\subset \R^n$ is the $i^{\textrm{th}}$ convex constraint set and can be rewritten as a set of convex inequality constraints and $m$ is the total number of inequality constraints required to represent the $q$ convex constraint sets. In \eqref{ftp_timeInvariant}, the point $X\in\R^n$ minimizes the sum-of-distances to each of the convex sets $\Omega_i$ and $X_i\in\Omega_i$ is the point corresponding to the minimal sum-of-distances on the $i^{\textrm{th}}$ convex set $\Omega_i$. In \cite{rejitha}, the authors proposed a nonsmooth projected primal-dual dynamical system to solve the eFTP \eqref{ftp_timeInvariant} in both centralized and distributed manner. The asymptotic convergence of the trajectories of the system proposed in \cite{rejitha} to the optimal points of the eFTP is guaranteed and proved by analyzing the semistability of the proposed nonsmooth dynamical system.
    
    In this work, we consider the TV version of the eFTP, by allowing the constraint sets to change with respect to time by preserving the convexity, closedness and nonempty feature of the sets. In other words, we solve the eFTP with TV nonempty, closed and convex sets. To ensure Assumption \ref{Assumption_StrongConvexity} holds, we approximate the objective function of \eqref{ftp_timeInvariant} with the sum-of-squared distances to the given TV convex sets and reformulate the TV eFTP as:
    \begin{equation}\label{TVeFTP}
        \begin{aligned}
            \min_{X(t),X_i(t)}\enspace&\sum_{i=1}^m \vectornorm{X(t)-X_i(t)}^2\\
        \textrm{s.t. }\enspace& g_i(X_i,t)\leq0,\quad i=1,2,\ldots,m.
        \end{aligned}
    \end{equation}
    We consider two different examples in this section to illustrate two different scenarios in a TV eFTP.  The first example contains time-varying convex sets as well as time-invariant convex sets and portraits the convergence of the system trajectories to the optimizer trajectories of the eFTP as time progresses. While the second example illustrates the important property of the proposed dynamical systems that allows the switching of the Lagrange multipliers from a zero optimal value to a positive optimal value during the active to inactive state switching of the corresponding inequality constraints. The examples are simulated in \texttt{MATLAB} with \texttt{ode} solvers. To verify the optimizer trajectory obtained through the proposed approaches, we generate the optimizer trajectory of the underlying TV convex optimization problem in a batch process using \texttt{CVX} toolbox in \texttt{MATLAB}. 
    \begin{table}[h!]
        \caption{Time varying constraint sets of Example \ref{Ex1}}
        \label{Tab1}
        \begin{tabular}{|c|c|c|ll}
            \cline{1-3}
                Set 
                    & \begin{tabular}[c]{@{}c@{}}Inequality constraint \\ at $t=t_0$\end{tabular}                                   & TV reference trajectory 
                        &  &  \\
            \cline{1-3}
                1  
                    & \begin{tabular}[c]{@{}c@{}}$\vectornorm{C_1(t)-X_1(t)}^2\leq 9$\\ with $ C_1(t_0) = (-12,12)$\end{tabular}      & $C_1(t)=\begin{pmatrix}0.4t\\-0.5t\end{pmatrix}$     &  &  \\ 
            \cline{1-3}
                2                  
                    & \begin{tabular}[c]{@{}c@{}}$\vectornorm{C_2(t)-X_2(t)}^2\leq 4$\\ with $ C_2(t_0) = (5,7)$\end{tabular}      & $C_2(t)=\begin{pmatrix}0.2t\cos (t)\\0.2t\sin(t)\end{pmatrix}$     &  &  \\ 
            \cline{1-3}
                3                  
                & \begin{tabular}[c]{@{}c@{}}$\vectornorm{C_3(t)-X_3(t)}^2\leq 2.25$\\ with $ C_3(t_0) = (7,-3)$\end{tabular} 
                    & $ C_3(t)=\begin{pmatrix}-0.8t\\-\frac{8}{6}(1+0.1t)\end{pmatrix}$            
                    &  &  \\ 
            \cline{1-3}
                4            
                    & \begin{tabular}[c]{@{}c@{}}$\vectornorm{C_4(t)-X_4(t)}^2\leq 9$\\ with $ C_4(t_0) = (-9,-13)$\end{tabular}      & $C_4(t)=\begin{pmatrix}6\sin (\frac{\pi}{25} t)\\6(\cos (\frac{\pi}{25} t)-1)\end{pmatrix}$ &  &  \\ 
            \cline{1-3}
                \multirow{2}{*}{5} 
                    & $(x_5+25)^2+(y_5-10)^2\leq 36$                    
                        & Time-invariant                                  &  &  \\ 
            \cline{2-3}
                  & $(x_5+25)^2+(y_5-15)^2\leq 25$                           & Time-invariant                                       &  &  \\ 
            \cline{1-3}
        \end{tabular}
    \end{table}
	\begin{example}\label{Ex1}
        Consider the example as shown in Fig. \ref{fig1_exp}, where the centre of the disks are continuously changing in time by following a specific TV path assigned to it. The set constraints and the corresponding TV pattern of the centre are provided in Table \ref{Tab1}. The centre of the disks-1 (blue) and disk-3 (green) moves linearly; disk-2 (magenta) moves spirally, disk-4 (yellow) moves in a circular path and the eye shaped set is stationary.  The evolution of the trajectories of the fixed-time convergent system \eqref{Dyn} with $\alpha=2$, $s_i(t)=\left(\vectornorm{g_i(x,t)}+\delta\right)e^{-t}$ where $\delta=0.1$ for $i=1,2,\ldots,m$ over time is shown in Fig.\ref{fig1_exp}. Figure \ref{fig_Traj_Exp} depicts the asymptotic convergence of its trajectories to the optimizer trajectory obtained with the \texttt{CVX} toolbox. The fixed-time dynamical system is implemented with parameters $c_1=1$, $c_2=1$, $\gamma_1=0.2$ and $\gamma_2=-2$ and which leads to an upper bound $T_{\max}$ of $5.6089$ seconds on the settling-time function. The trajectories of the fixed-time dynamical system converges to the optimizer trajectory before the upper-bound on the settling-time function as shown in Fig.\ref{fig_Traj_Fxts}. The slack variable for the fixed-time convergent system is chosen to depend on $T_{\max}$ as $s_i(t)=\left(\vectornorm{g_i(x,t)}+\rho\right)\left(1-e^{-(t-T_{\max})/t}\right)\left(1-\tanh(kt)\right)$, $\rho=0.001$ and $k=1$. The evolution of the dual optimizer trajectories of the asymptotic and fixed-time convergent system are shown in Fig. \ref{L1a} and Fig.\ref{L1b} respectively, which clearly shows the nonnegativity of the Lagrange multipliers over time. 
             \begin{figure}
        \centering
        \includegraphics[width=\linewidth]{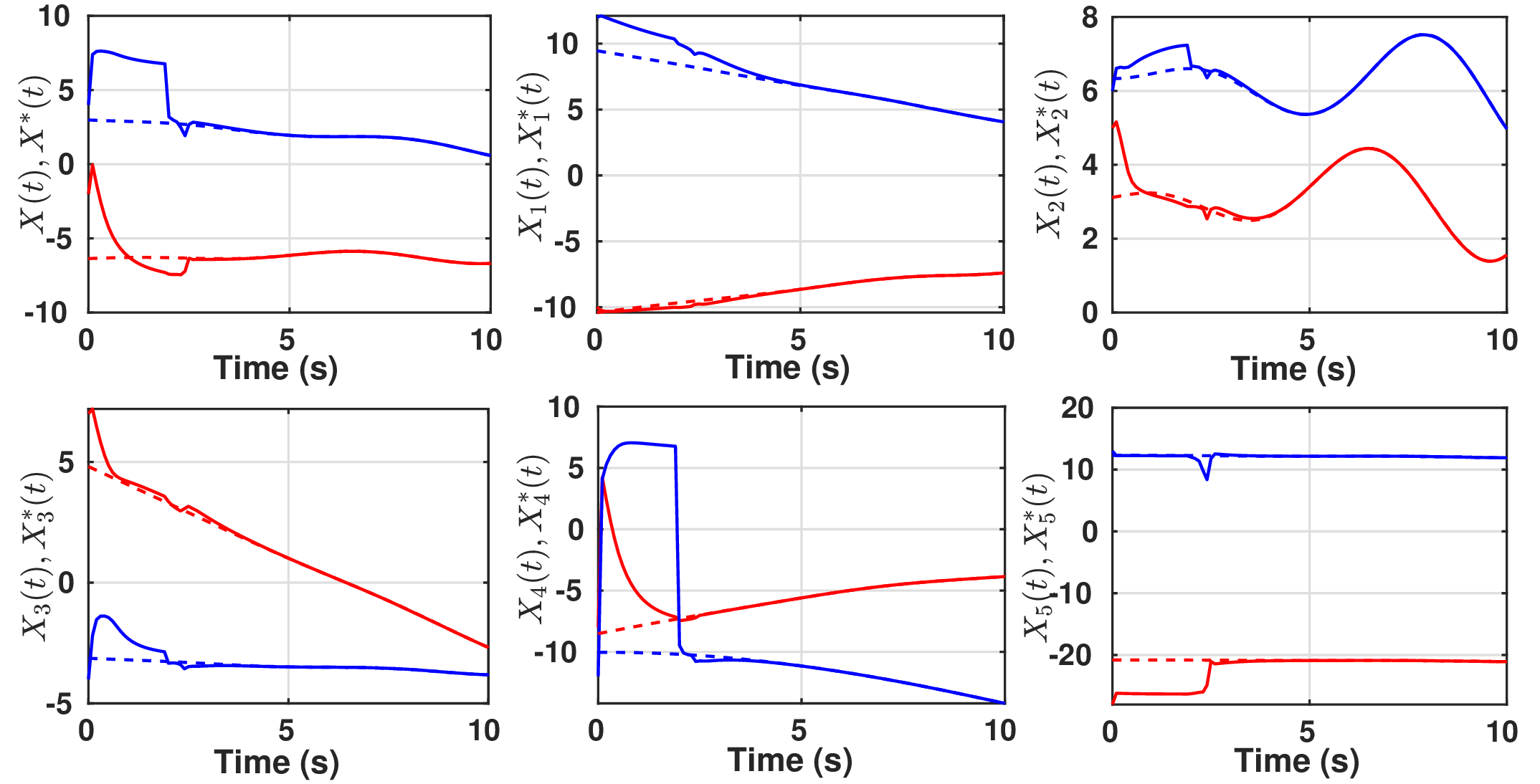}
        \caption{Example 1: Convergence of the trajectories of the system \eqref{Dyn} (solid line) to the optimizer trajectories (dashed line) obtained with \texttt{CVX}. The red and blue line indicates the $x$ and $y$ components of the trajectory.}
        \label{fig_Traj_Exp}
    \end{figure}
    \begin{figure}
        \centering
        \includegraphics[width=\linewidth]{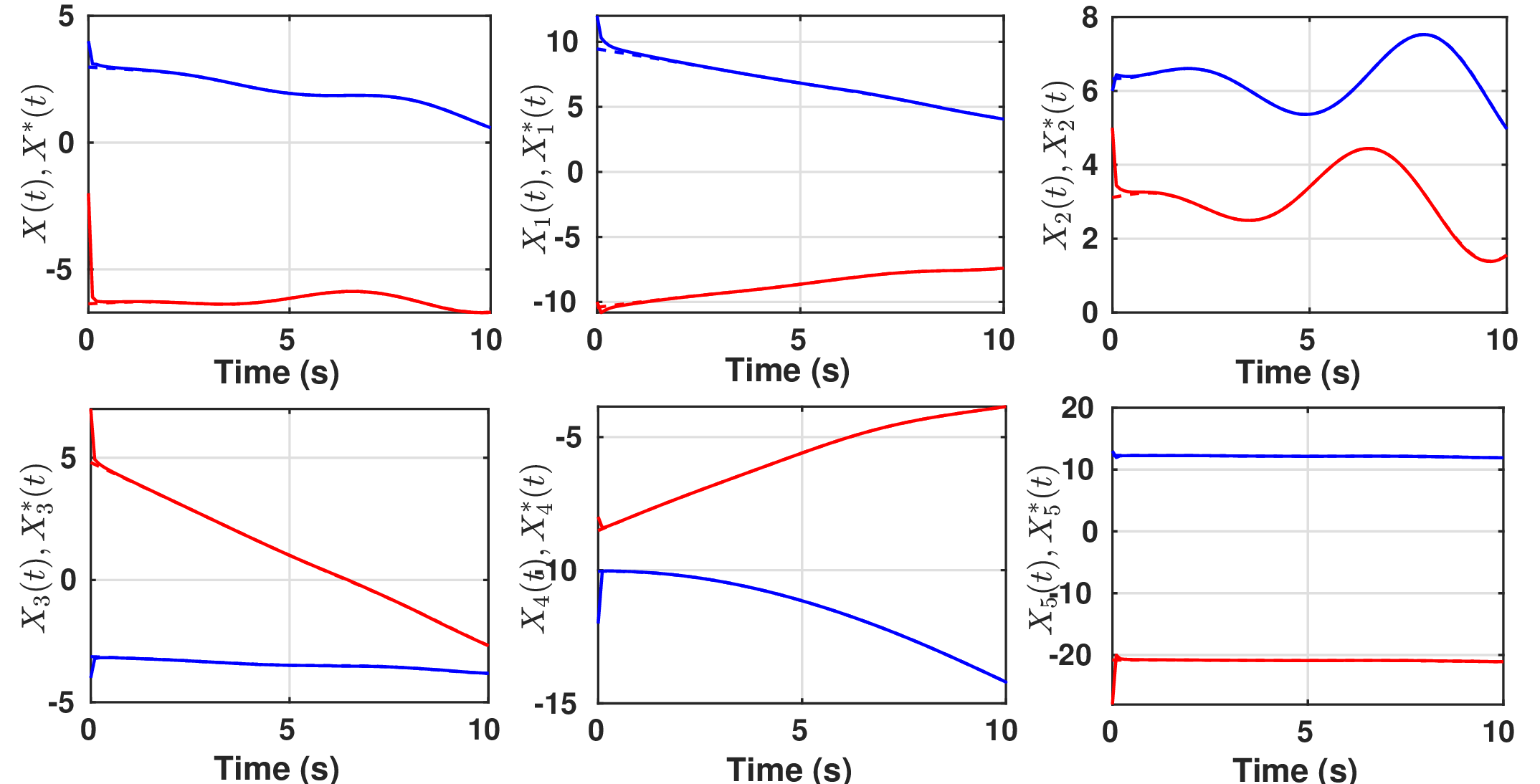}
        \caption{Example 1: Fixed-time convergence of the trajectories of the system \eqref{FxtsDyn} (solid line) to the optimizer trajectories (dashed line) obtained with \texttt{CVX}. The red and blue line indicates the $x$ and $y$ components of the trajectory.}
        \label{fig_Traj_Fxts}
    \end{figure}
     \begin{figure}
        \centering
        \subfloat[Asymptotically convergent system\label{L1a}]{\includegraphics[width=0.48\linewidth]{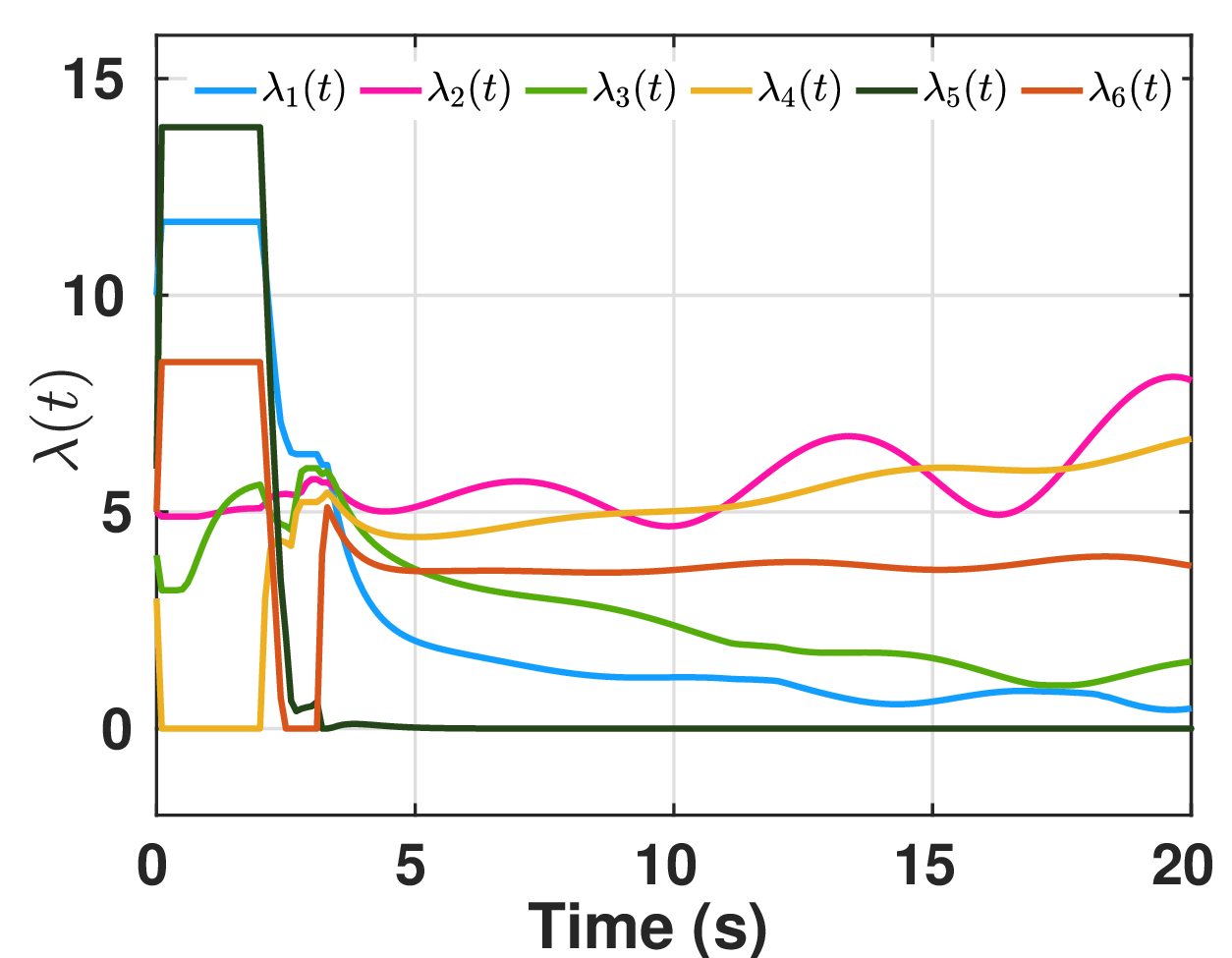}}
        \hfill
        \subfloat[Fixed-time convergent system\label{L1b}]{\includegraphics[width=0.48\linewidth]{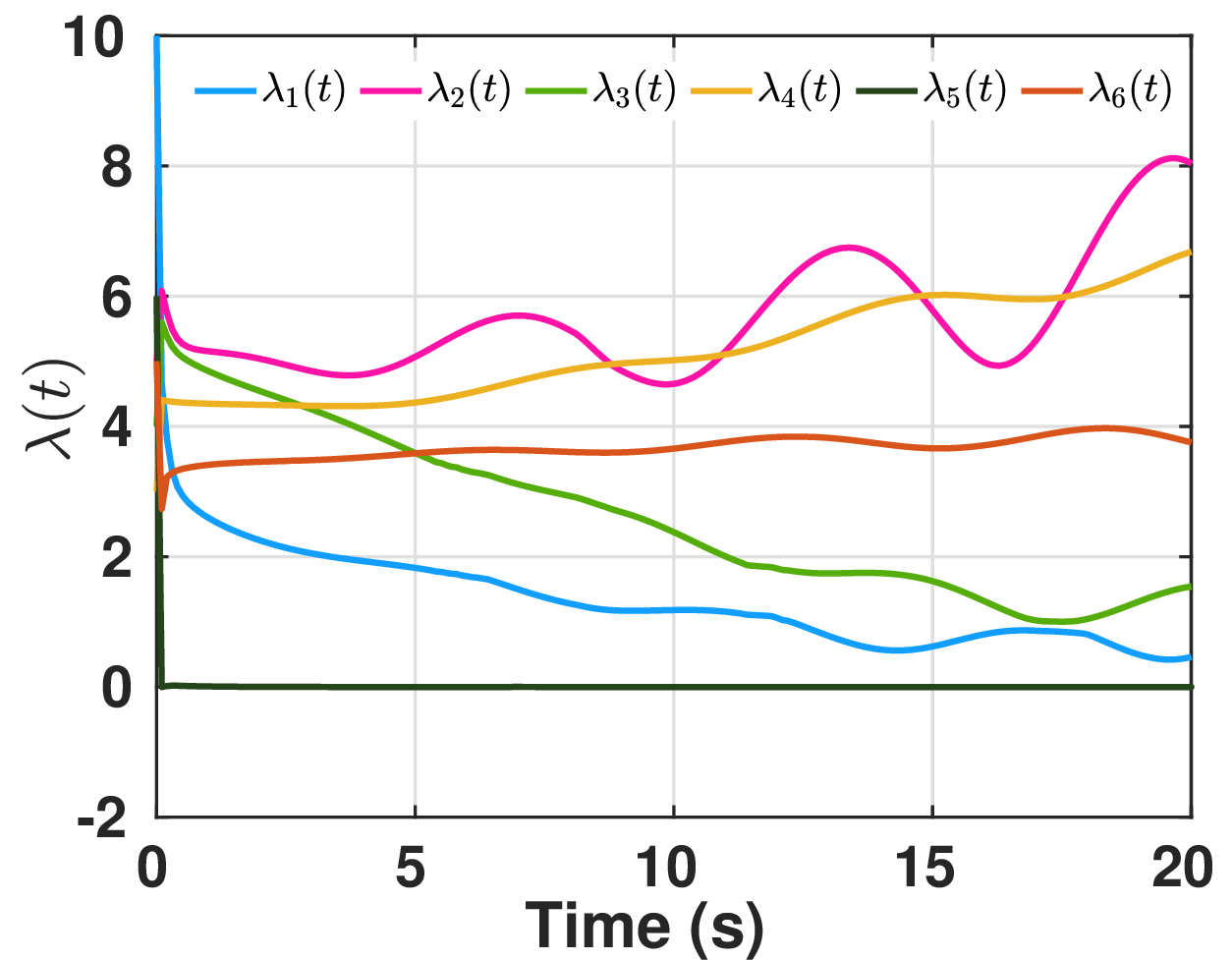}}
        \caption{Evolution of the dual trajectories associated with the Example \ref{Ex1}}
        \label{fig_LAm_Exp}
    \end{figure}
    \end{example}
    \begin{figure*}[t]
        \centering
        \includegraphics[width=\linewidth]{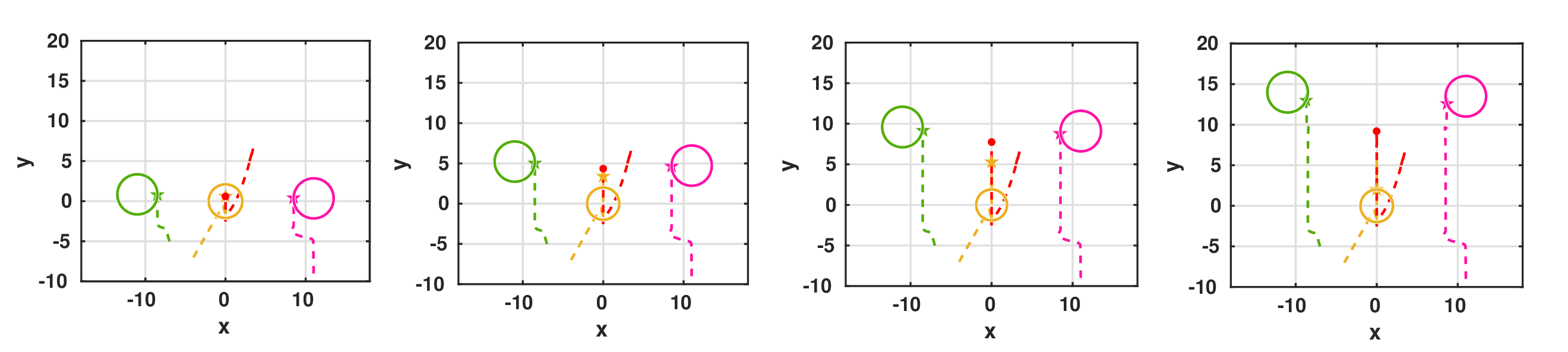}
        \caption{Example 2: Evolution of the trajectories of  \eqref{Dyn} with $\alpha=2$ over time. The red dotted line indicates the optimizer trajectory $x\st(t)$. Snapshots corresponds to the time instants $t=T_f/4,T_f/2,3T_f/4$ and $T_f$ with $T_f=70$ seconds.}
        \label{fig2_exp}
    \end{figure*}

    \begin{example}
        Consider three circular disks centred at (-11,-3.5), (0,0) and (11,-4) with radius 2.5 cm, 2 cm and 2.5 cm respectively as the TV convex sets at $t=0$. The convex sets are moving in such a way that the first and third disk moves along the positive y-axis, and the radius of the second disk oscillates about its initial value. That is, the centre of $\Omega_1$ and $\Omega_3$ follows the TV reference trajectory $C_i(t)=C_i(0)+(0, 0.25t)$ for $i=1,3$ and the radius of $\Omega_2$  follows the TV trajectory $r_2(t)=r_2(0)+0.1\sin(0.03\pi t)$. The TV nature of $r(t)$ allows the periodic expansion and contraction of the disk $\Omega_2$. This example portraits one of important feature of the proposed dynamical system, that allows the dual optimizer trajectory to escape from zero to a positive value. The evolution of the optimizer trajectory of the Example 2 over time is shown in Fig. \ref{fig2_exp}.  Initially, the convex sets are far apart and aligned in such a way that the centre of each disk lies on the vertex of a triangle. Therefore, the corresponding optimal points on each convex set lie on the boundary of the disk. Consequently, the inequality constraints are active, which leads to a positive Lagrange multiplier corresponding to each inequality constraint. As time progresses, the sets $\Omega_1(t)$ and $\Omega_3(t)$ start moving along the positive y-axis and align the disks in a straight line. Subsequently, the optimal points $X\st$ and $X_2\st$ move to the interior of the constraint set $\Omega_2$. 
            \begin{figure}[h!]
        \centering
        \includegraphics[width=\linewidth]{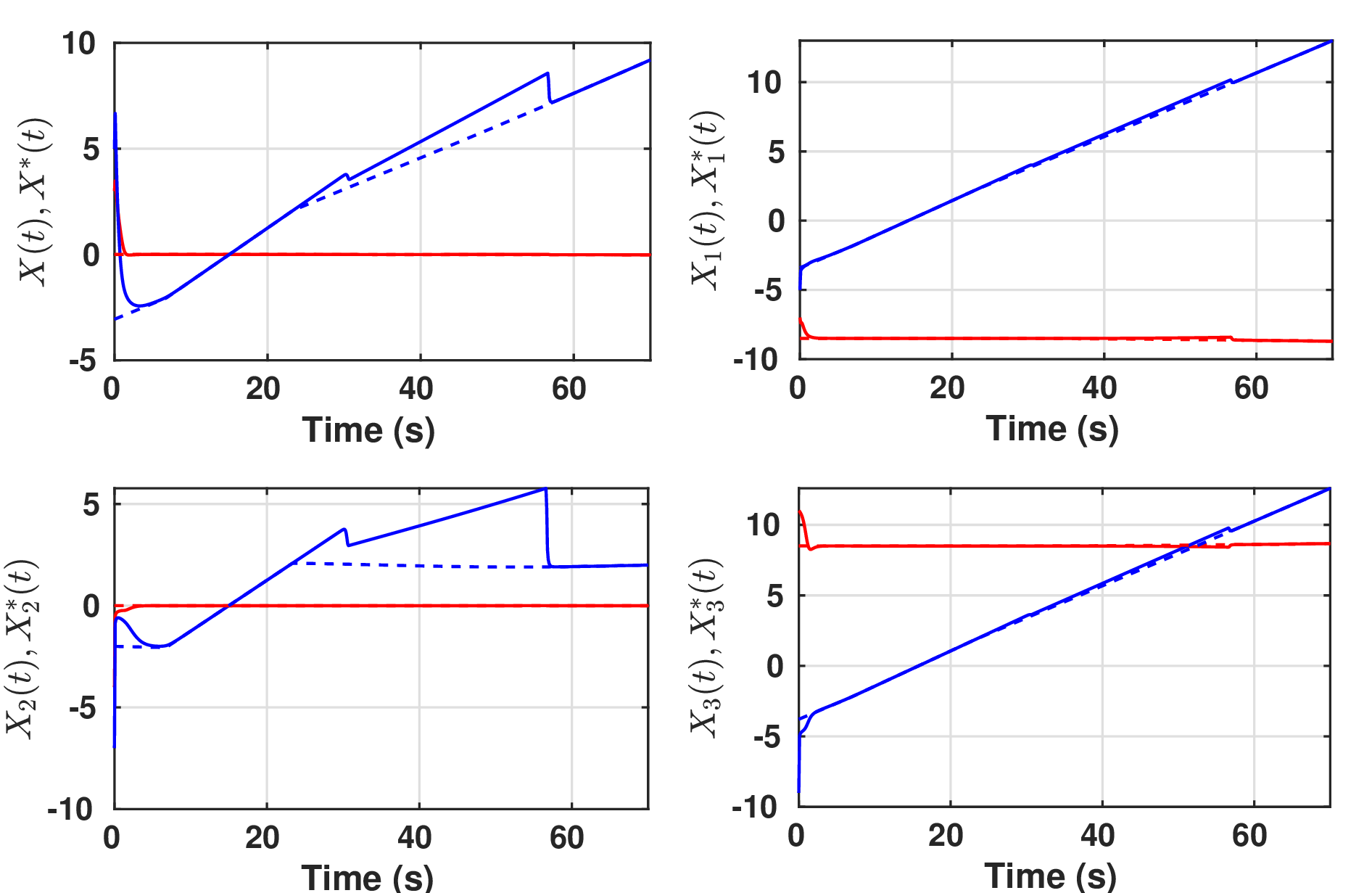}
        \caption{Example 2: Asymptotic convergence of the trajectories of the system \eqref{Dyn} (solid line) to the optimizer trajectories (dashed line) obtained with \texttt{CVX}. The red and blue line indicates the $x$ and $y$ components of the trajectory.}
        \label{fig2_Traj_Exp}
        \centering
        \includegraphics[width=\linewidth]{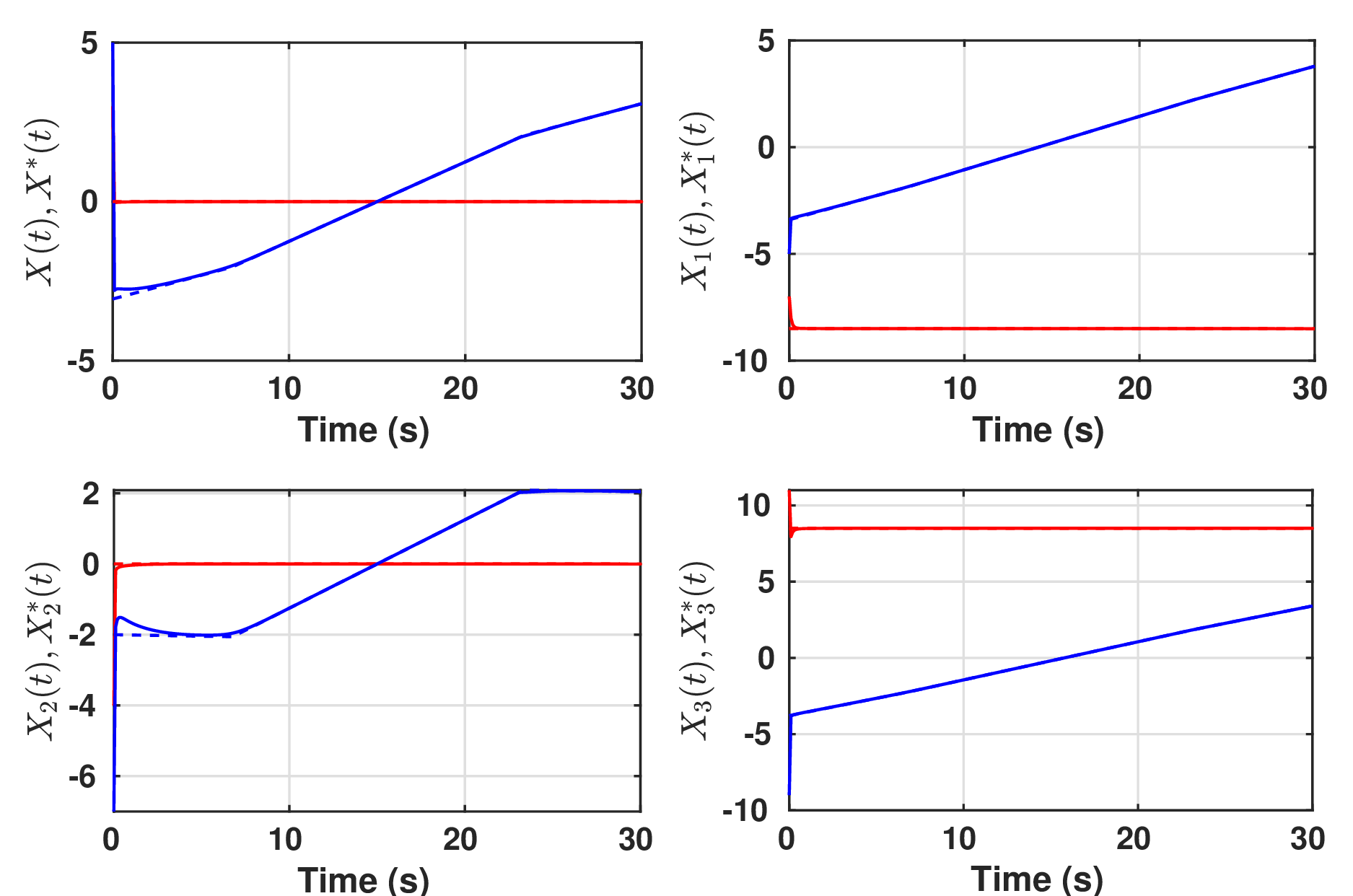}
        \caption{Example 2: Fixed-time convergence of the trajectories of the system \eqref{FxtsDyn} (solid line) to the optimizer trajectories (dashed line) obtained with \texttt{CVX}. The red and blue line indicates the $x$ and $y$ components of the trajectory.}
        \label{fig2_Traj_Fxts}
    \end{figure}
    \begin{figure}
        \centering
        \subfloat[Asymptotically convergent system\label{L2a}]{\includegraphics[width=0.48\linewidth]{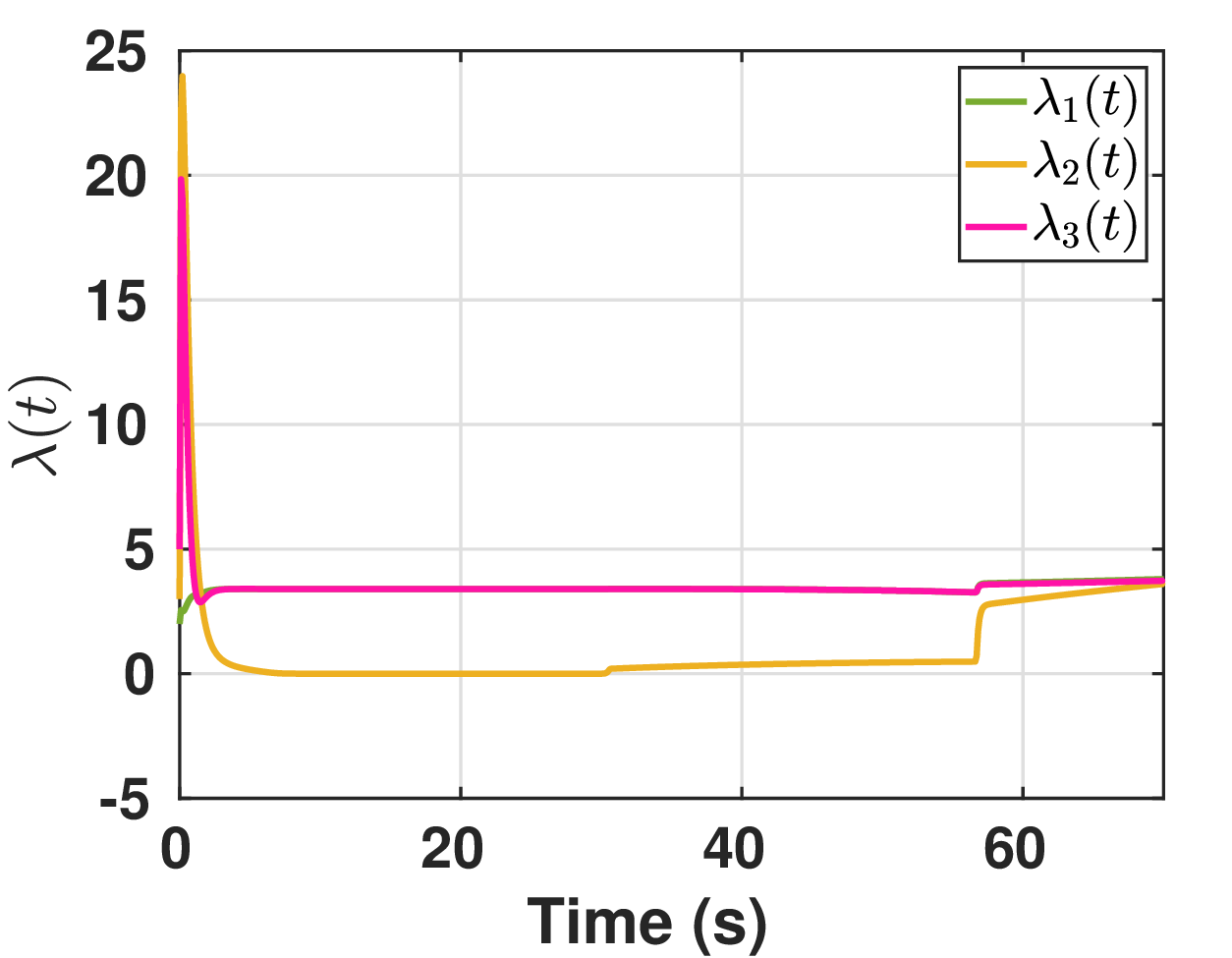}}
        \hfill
        \subfloat[Fixed-time convergent system\label{L2b}]{\includegraphics[width=0.48\linewidth]{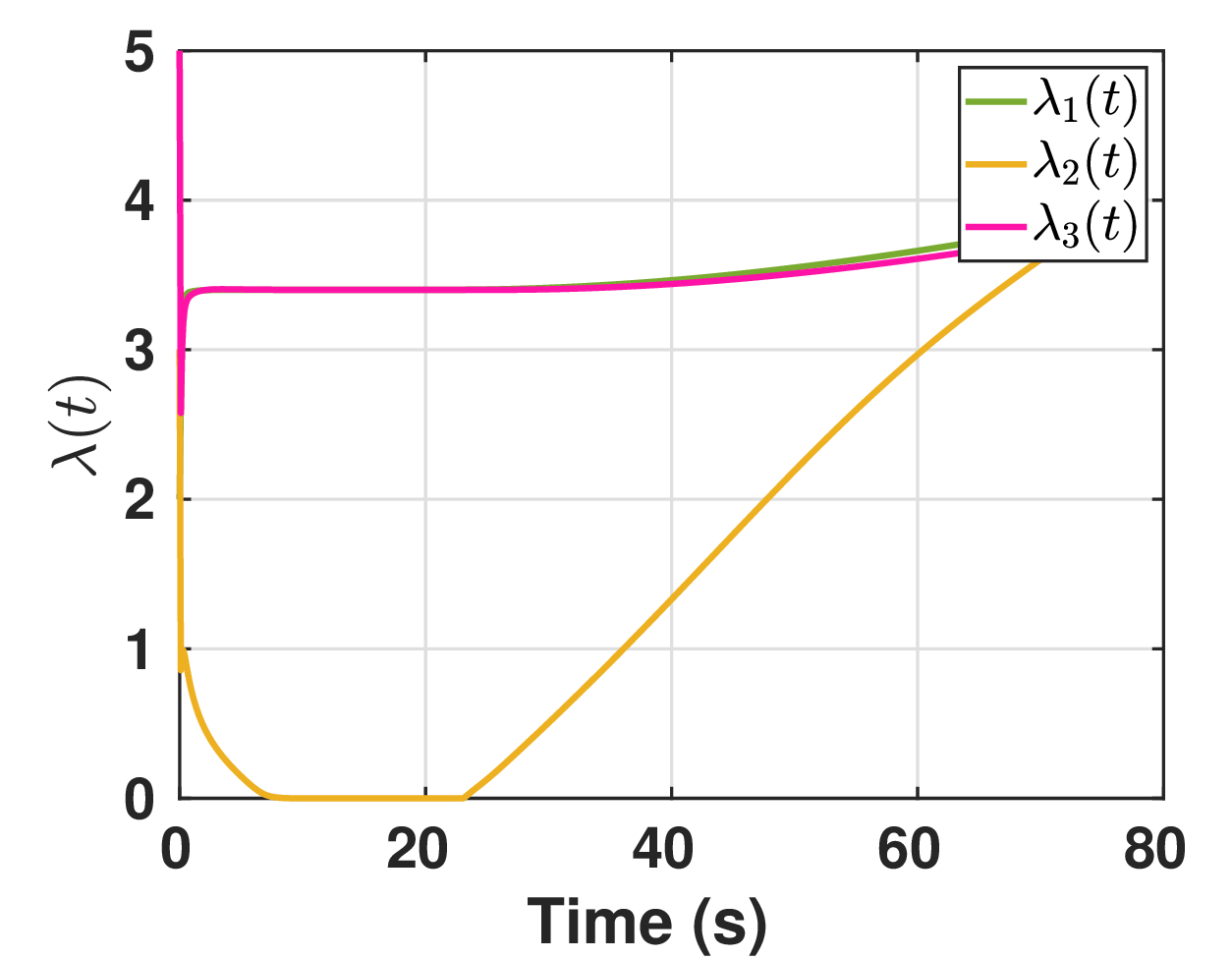}}
        \caption{Evolution of the dual optimizer trajectories associated with Example 2}
        \label{fig2_Lam_Fxts}
    \end{figure}
        Since the inequality constraint representing the convex set $\Omega_2$ is inactive in its interior, the corresponding optimal Lagrange multiplier becomes zero by the complementary slackness condition. As time advances,  the three disks again align triangularly, bringing the optimal point $X_2\st$ to the boundary of $\Omega_2$ and $\lambda_2\st$ to a positive value.  This transition of the dual optimizer trajectory between zero and a positive value is depicted in Fig. \ref{L2a} and Fig. \ref{L2b} for the asymptotic and fixed-time scenarios, respectively. 
        The parameters of the fixed-time system is chosen as, $c_1=1$, $c_2=1$, $\gamma_1=0.1$, $\gamma_2=-2$ and for slack variables $\rho=0.1$ and $k=0.001$ for $i=1,2,3$. We use the same slack variable from Example \ref{Ex1} for the asymptotically convergent system. Figure \ref{fig2_Traj_Exp} and Fig. \ref{fig2_Traj_Fxts} respectively shows the convergence of the trajectories of the asymptotic and fixed-time convergent systems to the optimizer trajectory obtained from the batch process. It is clear from the Fig. \ref{fig2_Traj_Exp} that there exists a transient phase during the switching of the Lagrange multiplier from zero to a positive value for the asymptotically convergent system. However this transient phase has been bypassed in the fixed-time dynamical system and its trajectories converge to the optimizer trajectory ahead of the upper bound of $10.6026$ seconds on the settling time function value as shown in Fig.\ref{fig2_Traj_Fxts}. 
        \end{example}
    \begin{table}[h!]
        \caption{Comparison of computation time among different approaches}
        \label{tab2}
        \centering
        \begin{tabular}{|l|ccc|}
        \hline
        \multirow{2}{*}{} & \multicolumn{3}{c|}{Computation time in seconds}\\    \cline{2-4} 
        & \multicolumn{1}{c|}{
            \begin{tabular}
                [c]{@{}c@{}}Batch Process\\ 
                using \texttt{CVX}
            \end{tabular}} 
        & \multicolumn{1}{c|}{
            \begin{tabular}
                [c]{@{}c@{}}Asymptotic\\ 
                Convergence
            \end{tabular}} 
        & \begin{tabular}
            [c]{@{}c@{}}Fixed-time\\ 
            Convergence
            \end{tabular} \\ 
        \hline
        Example 1 
            & \multicolumn{1}{c|}{332.9812}    
            & \multicolumn{1}{c|}{12.7220}                                            
            & 5.2017 \\ 
        \hline
        Example 2         
            & \multicolumn{1}{c|}{364.2198}                                         
            & \multicolumn{1}{c|}{42.3414}                                           & 0.5398   \\ 
        \hline
        \end{tabular}
`    \end{table}
    The computation time needed for the two examples with different approaches is tabulated in Table \ref{tab2}. It is validated from the Table \ref{tab2} that the batch process needs a larger computation time than the two dynamical systems approaches presented in this paper.  It substantiates that the proposed approaches performs better than the conventional batch approach to solve inequality constrained TV convex optimization problems.  Furthermore, among the two dynamical systems presented in this paper, the fixed-time dynamical system requires the least and the best computation time than the asymptotic dynamical system approach.
\section{Discussions and Future work}
    In this work we proposed projected dynamical systems to track the optimizer trajectory of an inequality constrained TV convex optimization problem with a strongly convex and twice continuously differentiable objective function and guaranteed the local asymptotic and fixed-time convergence of the proposed systems.
    A numerical example is employed to compare the performance of the proposed primal-dual prediction-correction dynamical system approaches with the existing prediction-correction approach. The results demonstrate that the proposed projected dynamical system outperforms the existing approach in terms of tracking capability. The proposed dynamical system approaches are later implemented to solve the approximated TV eFTP that minimizes the sum-of-squared distances to a finite number of nonempty, closed and TV convex sets. The approach presented in this work encourages a future scope of developing a nonsmooth projected dynamical system to solve unconstrained and constrained TV optimization problem by relaxing the twice differentiability and strong convexity of the objective function of TV convex optimization problems.
%%%%%%%%%%%%%%%%%%%%%%%%%%%%%%%%%%%%%%%%%%%%%%%%%%%%%%%%%%%%%%%%%%%%%%%%%
\appendices
\section{Proof of Lemma \ref{Lagragian_strongconvex}}\label{Proof_Lagrangian_StrongConvexity}
    \begin{proof}
        Consider  $x_1(t),x_2(t)\in\R^n$ and the Lagrange multiplier $\lambda=(\lambda_1(t),\lambda_2(t),\ldots,\lambda_m(t))\in\R^m_{\geq 0}$ at the time instant $t$. Then
            \begin{align*}
                L(x_2,&\lambda,t)-L(x_1,\lambda,t) \\
                &= f(x_2,t)-f(x_1,t)+\sum_{i=1}^m\lambda_i\left(g_i(x_2,t)-g_i(x_1,t)\right)\\
                &\geq \nabla_x f(x_1,t)^\top(x_2-x_1)+\frac{\mu}{2}\vectornorm{x_2-x_1}^2\\
                &+\sum_{i=1}^m\lambda_i\left(\nabla g_i(x_1,t)^\top(x_2-x_1)\right)\\
                &= \left(\nabla f(x_1,t)+\sum_{i=1}^m \lambda_i\nabla g_i(x_1,t)\right)^\top(x_2-x_1)\\
                &+\frac{\mu}{2}\vectornorm{x_2-x_1}^2\\
                &= \nabla_x L(x_1,\lambda,t)^\top(x_2-x_1)+\frac{\mu}{2}\vectornorm{x_2-x_1}^2.
            \end{align*}
        The inequality in the proof follows from the strong convexity of the objective function $f(x,t)$ in $x$ for all $t\geq0$.
    \end{proof}
    
%%%%%%%%%%%%%%%%%%%%%%%%%%%%%%%%%%%%%%%%%%%%%%%%%%%%%%%%%%%%%%%%%%%%%%%%%
\section{Proof of Lemma \ref{Lemma_Invertibility}}\label{Proof_Lemma_Invertibility}
    \begin{proof}
        The invertibility of the time-invariant $J$ matrix can be found in  \cite{bertsekas2014constrained}.
        We provide an alternate proof for  the invertibility of  $J(x\st,\lambda\st,t)$  
 by showing the invertability of Schur complement of $\xx L$. Consider the following three distinct scenarios:
        \begin{enumerate}[(i)]
            \item $g_i(x\st,t)<0$ for $i=1,2,\ldots,m$ at time $t$.\\
                Then by complementary slackness condition \eqref{KKT2},  $\lambda_i(t)=0$ for all $i=1,2,\ldots,m$ and thus $M=G_d(x\st,t)\prec0$.
            \item $\lambda_i(x\st,t)>0$ for $i=1,2,\ldots,m$.\\
                Then $G_d=0_{m\times m}$ by \eqref{KKT2} and the Schur matrix $M=-\lambda\st\circ\x G^\top\xx L^{-1}\x G\prec0$ by Assumption \ref{Assumption_linearindependence}.
            \item $\lambda_i\st(t)>0$ for $i=1,2,\ldots, \vert I(x,t)\vert$ and $\lambda_i\st(t)=0$ for $i=\vert I(x,t)\vert +1,\ldots,m$ with $\\vert I(x,t)\vert <m$. \\
                Then $g_i(x\st,t)<0$ for $i=\vert I(x,t)\vert+1,\ldots,m$ by Assumption \ref{Assumption_StrictComplementary} and it further guarantees the non singularity of the Schur matrix $M$ under  Assumption \ref{Assumption_linearindependence}. It can be verified from the structure of $M(x,\lambda,t)$ matrix given in \eqref{MMatrix}.
        \begin{figure*}[b]
            \begin{equation}\label{MMatrix}
                M=
                \begin{bmatrix}
                    g_1-\lambda_1(t)\x g_1^\top\xx L^{-1}\x g_1 &-\lambda_1(t)\x g_1^\top\xx L^{-1}\x g_2&\hdots&-\lambda_1(t)\x g_1^\top\xx L^{-1}\x g_m\\
                    -\lambda_2(t)\x g_2^\top\xx L^{-1}\x g_1&g_2-\lambda_2(t)\x g_2^\top\xx L^{-1}\x g_2&\hdots&-\lambda_2(t)\x g_2^\top\xx L^{-1}\x g_m\\
                    \vdots&\vdots&\ddots&\vdots\\
                    -\lambda_m(t)\x g_m^\top\xx L^{-1}\x g_1&-\lambda_m(t)\x g_m^\top\xx L^{-1}\x g_2&\hdots&g_m-\lambda_m(t)\x g_m^\top\xx L^{-1}\x g_m
                \end{bmatrix}
            \end{equation}
        \end{figure*}
    \end{enumerate}
    Thus the existence of the inverse of the block matrix $J(x\st,\lambda\st,t)$ is guaranteed for all $t\geq0$ and is,
        \begin{equation}
            J^{-1}(x\st,\lambda\st,t)=
                \begin{bmatrix}
                    A & B\\ C & D
                \end{bmatrix}
        \end{equation}
        \begin{align*}
            &\textrm{where }A = \xx L^{-1}\left(I_n+\x GM^{-1}\left(\lambda\st\circ\x G^\top\right)\xx L^{-1}\right)\\
            &B = -\xx L^{-1}\x GM^{-1},~
            C = -M^{-1}\left(\lambda\st\circ\x G^\top\right)\xx L^{-1}\\
            &\textrm{ and }D = M^{-1}.
        \end{align*}
    \end{proof}
%%%%%%%%%%%%%%%%%%%%%%%%%%%%%%%%%%%%%%%%%%%%%%%%%%%%%%%%%%%%%%%%%%%%%%%%%
\section{Proof of Lemma \ref{Lemma_StrongConvexity}}\label{Proof_Lemma_StrongConvexity}
    \begin{proof}
        The assumptions made throughout this paper hold for each time instant $t$ and hence the proof hold for all $t\geq 0$.
        We have to prove that $\x L(x,\lambda,t)=0$ if and only if $(x(t),\lambda(t))=(x\st(t),\lambda\st(t))$ for all $t\geq 0$. The KKT condition \eqref{KKT1} ensures that that if $(x(t),\lambda(t))=(x\st(t),\lambda\st(t))$, then $\x L(x,\lambda,t)=0$. To prove the converse, let us assume that there exists a feasible, non-optimal primal-dual pair $(\tilde{x},\tilde{\lambda})$ at time $t$ such that $\x L(\tilde{x},\tilde{\lambda},t)=0$. Since the Lagrangian function is strongly convex $\x L(\tilde{x},\tilde{\lambda},t)=0$ implies,
        \begin{equation}\label{Eq1}
            \inf_{x(t)} L(x,\tilde{\lambda},t)=L(\tilde{x},\tilde{\lambda},t).
        \end{equation}
        Now consider the dual optimization problem 
        \begin{equation}\label{Eq2}
            \max_{\lambda(t)\geq0} \enspace Q(\lambda,t)
        \end{equation}
        where $Q(\lambda,t)$ is the Lagrange dual function defined as
        \begin{equation}\label{Eq3}
            Q(\lambda,t)=\inf_{x(t)} L(x,\lambda,t).
        \end{equation}
        Since $\tilde{\lambda}(t)$ is not the dual optimal trajectory,
        \begin{align}\label{Eq4}
            Q(\tilde{\lambda},t)  \leq Q(\lambda\st,t)
        \end{align}
        where
        \begin{equation}\label{Eq5}
            Q(\tilde{\lambda},t)=\inf_{x(t)} L(x,\tilde{\lambda},t)=L(\tilde{x},\tilde{\lambda},t)
        \end{equation}
        and 
        \begin{equation}\label{Eq6}
            Q(\lambda\st,t)=\inf_{x(t)} L(x,\lambda\st,t)=L(x\st,\lambda\st,t).
        \end{equation}
        Equation \eqref{Eq5} follows from \eqref{Eq1}, and \eqref{Eq6} from the saddle point property $L(x\st,\lambda,t)\leq L(x\st,\lambda\st,t)< L(x,\lambda\st,t)$. On substituting \eqref{Eq5} and \eqref{Eq6} into \eqref{Eq4}, we get
        \begin{equation}\label{32}
            L(\tilde{x},\tilde{\lambda},t)\leq L(x\st,\lambda\st,t).
        \end{equation}
        By Assumption \ref{Assumption_StrongConvexity}, the primal optimal point $x\st$ is unique for all $t\geq0$, and under the Assumption \ref{Assumption_linearindependence}, the KKT condition \eqref{KKT1} guarantees the uniqueness of the dual optimal point $\lambda\st$ at each time $t\geq0$. Since $\x L(x\st,\lambda\st,t)=0$,  for \eqref{32} to hold true, the feasible point $(\tilde{x},\tilde{\lambda})$ has to be same as the primal-dual optimal pair $(x\st,\lambda\st)$ for all $t\geq0$. Since it is true for all $t\geq 0$, the feasible trajectory $(\tilde{x}(t),\tilde{\lambda}(t))$ coincides with the optimizer trajectory $(x\st(t),\lambda\st(t))$. That is, the gradient of the Lagrangian function vanishes only along the unique primal-dual optimizer trajectory $(x\st(t),\lambda\st(t))$ of the problem \eqref{Ineq}. 
    \end{proof}
%%%%%%%%%%%%%%%%%%%%%%%%%%%%%%%%%%%%%%%%%%%%%%%%%%%%%%%%%%%%%%%%%%%%%%%%%
\section{Proof of Lemma \ref{lemmaAdditionalTerm}}\label{lemmaAdditionalTerm_Proof}
    \begin{proof}
        The proof of necessity follows from the KKT optimality condition \eqref{KKT1}.
        To prove the sufficiency, consider $\lambda(t)=0$ at some time instant $\Bar{t}\geq 0$, then \eqref{Lambda_Dynamics} reduces to $\X_{\lambda}=\x G^\top(x,t)\x L(x,\lambda,t)$ for all $t\geq\Bar{t}$. That is, $\lambda(t)$ escapes from zero unless $\x G^\top(x,t)\x L(x,\lambda,t)=0$ for all $t\geq \Bar{t}$. Then by Assumption \ref{inactiveGi},
        $\x G^\top(x,t)\x L(x,\lambda,t)\equiv0$ can hold only if $\x L(x,\lambda,t)=0$ for all $t\geq \Bar{t}$. It implies by Lemma \ref{Lemma_StrongConvexity} that the augmented term vanishes only along the primal-dual optimizer trajectory $\left(x\st(t),\lambda\st(t)\right)$.
    \end{proof}
%%%%%%%%%%%%%%%%%%%%%%%%%%%%%%%%%%%%%%%%%%%%%%%%%%%%%%%%%%%%%%%%%%%%%%%%%

\section{Proof of Lemma \ref{ProjectedSystem}}\label{Lemma_ProjectedSystem}
    \begin{proof}
        The proof of Lemma \ref{ProjectedSystem} is along the lines of \cite[Lemma 4.2]{Cherukuri:asymp_conv}. Consider the vector field $\X:\R^n\times\R^m\times\R_{\geq0}\rightarrow\R^n\times\R^m$.
        There exists a $\tilde{\delta}$ such that for all $\delta\in[0,\tilde{\delta})$ and for $i=1,2,\ldots,m$, 
        \begin{equation}
            \begin{aligned}
                \textrm{proj}&_{\K}\left((x(t),\lambda(t))+\delta \X(x(t),\lambda(t)))\right) \\
                &=\begin{cases}
                    (x(t),\lambda(t))+\delta\X(x,\lambda,t),  \quad\;
                    \textrm{ if } (x(t),\lambda_i(t))\notin\Xi_i(t)\\
                    (x(t),\lambda(t))+\delta
                    \begin{pmatrix}
                       \X_x\\\left(\X_{\lambda}\right)_1\\\vdots\\\left(0\right)_i\\\vdots\\\left(\X_{\lambda}\right)_m
                    \end{pmatrix},
                    \textrm{ if } (x(t),\lambda_i(t))\in\Xi_i(t)
                \end{cases}
            \end{aligned}
        \end{equation}
        where $\K=\R^n\times\R^m_{\geq0}$. Then the associated projected dynamical system corresponding to $\X(x,\lambda,t)$ at $(x(t),\lambda(t))$ is
        \begin{align}
            \Pi_{\K}\big((x(t),\lambda(t))&,\X(x,\lambda,t)\big) \nonumber\\
            &=\begin{cases}
                \begin{pmatrix}
                    \X_x\\\left(\X_{\lambda}\right)_1\\\vdots\\\left(0\right)_i\\\vdots\\\left(\X_{\lambda}\right)_m
                    \end{pmatrix},
                    \hspace{-0.5cm}&\textrm{ if } (x(t),\lambda_i(t))\in\Xi_i(t)
                \\
                \X(x,\lambda,t), &\textrm{otherwise}.
            \end{cases}
        \end{align}
        This implies that $\X_{pd}=\Pi_{\K}\left((x,\lambda,t), \X(x,\lambda,t)\right)$ for all $(x,\lambda,t)\in\K\times\R_{\geq0}$.
    \end{proof}
%%%%%%%%%%%%%%%%%%%%%%%%%%%%%%%%%%%%%%%%%%%%%%%%%%%%%%%%%%%%%%%%%%
\section{Proof of Lemma \ref{Lemma_ContinuityofFXTS}}\label{Proof_Lemma_Continuity}
    \begin{proof}
        The proof is along the lines of the proof of \cite[Lemma 6]{Fxts:Garg}. 
        Let $z\st(t)=(x\st(t),\lambda\st(t))$. 
        To show the continuity of the right hand side of \eqref{FXtsVectorField} at $z(t)=z\st(t)$ for all $t\geq 0$, substitute $z(t)=z\st(t)$ in \eqref{FXtsVectorField}. Then we get,
        \begin{align*}
            \dot{z}(t)|_{z\st(t)}=& -\tilde{J}^{-1}(z\st,t)\left(\dfrac{c_1\x L(z\st,t)}{\vectornorm{\x L(z\st,t)}^{\gamma_1}}+\dot{z}\st(t)\right),
        \end{align*}
        $\tilde{J}^{-1}\rightarrow J^{-1}$ as $z(t)\rightarrow z\st(t)$.
        Now we need to show that $\lim_{z(t)\rightarrow z\st(t)} \dfrac{\x L(z,t)}{\vectornorm{\x L(z,t)}^{\gamma_1}}=0$. 
        A trajectory $z(t)$ of \eqref{FxtsDyn} is said to converge to the trajectory $z\st(t)$ (denoted by $z(t)\rightarrow z\st(t)$) in the norm sense, if for a given $\epsilon>0$, there exists $\Bar{T}>0$ such that for all $t\geq \Bar{T}\geq 0$, $\vectornorm{z(t)-z\st(t)}<\epsilon$. Since the norm function is continuous, under the Assumption \ref{Assumption_boundedness},
        \begin{align*}
            &\left\|\lim_{\vectornorm{z(t)-z\st(t)}\rightarrow0}\dfrac{\x L(z,t)}{\vectornorm{\x L(z,t)}^{\gamma_1}}\right\|\\
            &\;\;=\lim_{\vectornorm{z(t)-z\st(t)}\rightarrow0}\left\|\dfrac{\x L(z,t)}{\vectornorm{\x L(z,t)}^{\gamma_1}}\right\|\\
            &\;\;= \lim_{\vectornorm{z(t)-z\st(t)}\rightarrow 0}\left\|\x L(z,t)-\x L(z\st,t)\right\|^{1-\gamma_1}\\
            &\;\;\leq\mathcal{L}_p^{1-\gamma_1}\lim_{\vectornorm{z(t)-z\st(t)}\rightarrow0}\vectornorm{z(t)-z\st(t)}^{1-\gamma_1}=0.
        \end{align*}
        This implies that the right hand side of \eqref{FXtsVectorField} is continuous at $z(t)=z\st(t)$ for all $t\geq 0$ and hence it is continuous for all $x(t)\in\R^n$.
    \end{proof}

\section*{Acknowledgment}
The work of Arun D. Mahindrakar was supported by SERB under the matrix project MTR\slash2020\slash000474 while the work of Umesh Vaidya was supported by NSF ECCS grant 2031573.

\bibliographystyle{IEEEtran}
\bibliography{IEEEabrv,ref}

\begin{IEEEbiography}[{\includegraphics[width=1in,height=1.25in,clip,keepaspectratio]{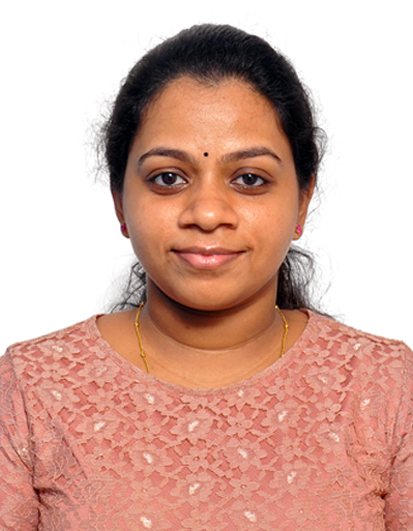}}]{Rejitha Raveendran} received the B.Tech degree in electronics and communication engineering, from University of Kerala, India in 2014, M.Tech degree in control systems from Indian Institute of Space Science and Technology, Thiruvananthapuram, India in 2017 and PhD degree from Indian Institute of Technology madras, India in 2022. She is currently a postdoctoral researcher at the Research and Innovation group of Tata Consultancy Services since November 2022.

Her research are include convex optimization, control systems, stability of nonlinear dynamical systems and electric mobility.
\end{IEEEbiography}

\begin{IEEEbiography}[{\includegraphics[width=1in,height=1.25in,clip,keepaspectratio]{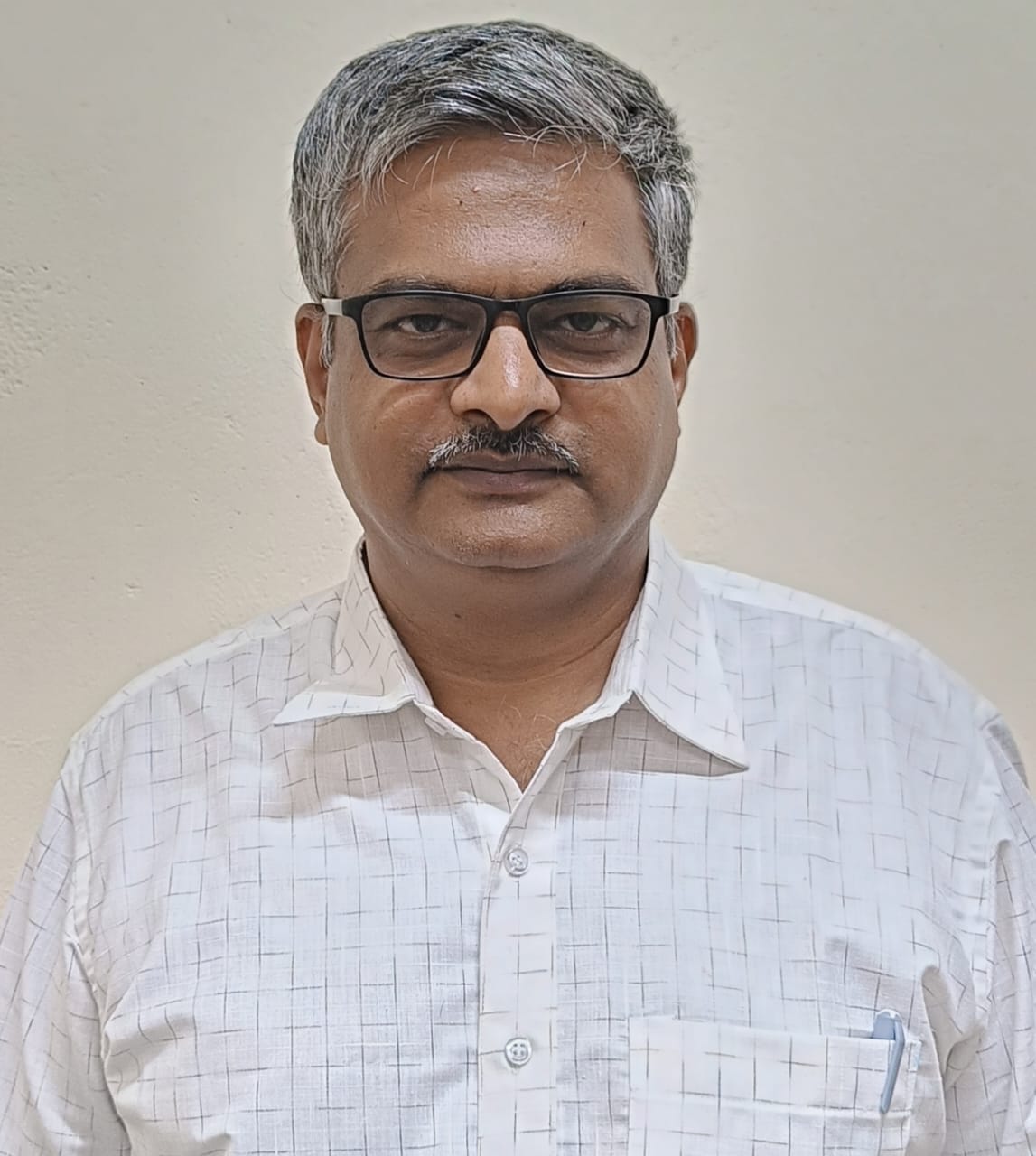}}]{Arun D Mahindrakar} received the Bachelor of Engineering degree in electrical and electronics from Karnatak University, Dharwad, India, in 1994, the Master’s of Engineering degree in control systems from Veermata Jijabai Technological Institute, Mumbai, India, in 1997, and the Ph.D. degree in systems and control from Indian Institute of Technology Bombay, India, in 2004. He was a Postdoctoral Fellow with the Laboratory of Signals and Systems, Supelec, Paris, France, from 2004 to 2005. 

He is currently a Professor with the Department of Electrical Engineering, Indian Institute of Technology Madras, Chennai, India. His research interests include nonlinear stability, geometric control, formation control of multiagent systems and convex optimization.
\end{IEEEbiography}

\begin{IEEEbiography}[{\includegraphics[width=1in,height=1.25in,clip,keepaspectratio]{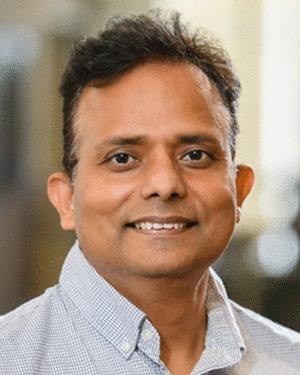}}]{Umesh Vaidya}  received the Ph.D. degree in mechanical engineering from the University of California at Santa Barbara, Santa Barbara, CA, in 2004. He was a Research Engineer with the United Technologies Research Center (UTRC), East Hartford, CT, USA. He is currently a Professor with the Department of Mechanical Engineering, Clemson University, S.C., USA. Before joining Clemson University in 2019, and since 2006, he was a Faculty with the Department of Electrical and Computer Engineering, Iowa State University.

His current research interests include dynamical systems and control theory. He was the recipient of 2012 National Science Foundation CAREER award.

\end{IEEEbiography}
\end{document}